\documentclass{amsart}
\usepackage[utf8]{inputenc}

\usepackage{csquotes}

\usepackage{amsmath}
\usepackage{amsfonts}
\usepackage{mathtools}
\usepackage{amssymb,latexsym,stmaryrd}
\usepackage{tikzsymbols}

\usepackage{enumitem}

\usepackage{tikz}
\usetikzlibrary{matrix,arrows,calc,external}

\usepackage{calc}

\usepackage[mode=multiuser,status=draft,lang=french]{fixme}
\fxsetup{theme=colorsig}

\usepackage{hyperref}
\hypersetup{final,colorlinks=true,
	linkcolor=[rgb]{0.1, 0.1, 0.44},
	citecolor=[rgb]{0.1, 0.1, 0.44}}

\usepackage{amsthm}

\usepackage[capitalise]{cleveref}

\usepackage{marginnote}

\usepackage[style=alphabetic, bibstyle=alphabetic, backend=bibtex, hyperref=true,url=false,doi=true,isbn=false]{biblatex}

\title{Embeddability on functions: order and chaos}

\author{Rapha\"{e}l Carroy}
\email{ raphael.carroy@univie.ac.at}
\address{Kurt Gödel Research Center for Mathematical Logic, Universität Wien, Währinger Strasse 25, 1090 Wien, Austria}

\author{Yann Pequignot}
\email{pequignot@math.ucla.edu}
\address{Department of Mathematics
	University of California, Los Angeles
	Box 951555
	Los Angeles, CA 90095-1555
	USA}

\author{Zolt\'an Vidny\'anszky}
\email{ zoltan.vidnyanszky@univie.ac.at}
\address{Kurt Gödel Research Center for Mathematical Logic, Universität Wien, Währinger Strasse 25, 1090 Wien, Austria}

\thanks{The first and third author were supported respectively by FWF Grants P28153 and P29999. The second author gratefully acknowledges the support of the Swiss National Science Foundation (SNF) through grant P2LAP2$\_$164904. The third author was also partially supported by he National Research, Development and Innovation Office -- NKFIH,  grants  no.  113047  and  104178.}

\FXRegisterAuthor{Y}{aY}{Y} 
\FXRegisterAuthor{R}{aR}{R} 
\FXRegisterAuthor{Z}{aZ}{Z} 

\theoremstyle{plain}
\newtheorem{theorem}{Theorem}[section]
\newtheorem{lemma}[theorem]{Lemma}
\newtheorem{fact}[theorem]{Fact}

\newtheorem{corollary}[theorem]{Corollary}
\newtheorem{conjecture}[theorem]{Conjecture}
\newtheorem{proposition}[theorem]{Proposition}
\newtheorem*{theorem*}{Theorem}

\newtheorem*{claim*}{Claim}
\newtheorem*{proposition*}{Proposition}

\theoremstyle{definition}
\newtheorem{definition}[theorem]{Definition}

\newtheorem{question}[theorem]{Question}

\theoremstyle{remark}
\newtheorem{remark}[theorem]{Remark}

\newtheorem*{definition*}{Definition}

\newcommand{\zerodim}{$0\text{-}$dimen\-sional}

\DeclareMathOperator{\proj}{proj}

\newcommand{\wqo}{\textsc{wqo}}
\newcommand{\bqo}{\textsc{bqo}}

\newcommand{\N}{{\mathbb N}}

\newcommand{\Q}{\mathbb Q}

\newcommand{\dom}{\operatorname{dom}}
\newcommand{\rest}[1]{ |_{#1}}

\newcommand{\G} {\mathbb{G}}

\newcommand{\Gr}{\mathcal{G}}

\newcommand{\rao}{\rightarrow}
\newcommand{\lrao}{\leftrightarrow}

\DeclareMathOperator{\diam}{diam}

\newcommand{\IFF}{\quad\longleftrightarrow\quad}

\newcommand{\sub}{\subseteq}
\newcommand{\emb}{\sqsubseteq} 
\newcommand{\ihom}{\preccurlyeq} 

\newcommand{\bijc}[1]{\left\langle #1 \right\rangle}

\DeclareMathOperator{\im}{im} 

\newcommand{\C}{C} 
\newcommand{\LC}{LC} 
\newcommand{\CCG}{\mathbb{G}} 

\DeclareMathOperator{\isolated}{Is} 
\DeclareMathOperator{\CB}{CB} 

\newcommand{\lin}{Lin_\text{c}} 
\newcommand{\POD}{P0D} 

\newcommand{\QOcolor}{\mathrm{L}}

\newcommand{\ana}{\mathbf{\Sigma}^1_1}

\crefformat{equation}{#2(#1)#3}
\crefformat{enumi}{#2(#1)#3}
\Crefformat{enumi}{#2(#1)#3}
\crefname{enumi}{}{}

\begin{document}
	
	\begin{abstract}

		We study the quasi-order of topological embeddability on definable functions between Polish \zerodim{} spaces.
		
		We first study the descriptive complexity of this quasi-order restricted to the space of continuous functions. Our main result is the following dichotomy: the embeddability quasi-order restricted to continuous functions from a given compact space to another is either an analytic complete quasi-order or a well-quasi-order.

		
		We then turn to the existence of maximal elements with respect to embeddability in a given Baire class. 
		We prove that no Baire class admits a maximal element, except for the class of continuous functions which admits a maximum element.   
		
	\end{abstract}
	
	\maketitle

	\tableofcontents

	
	\section{Introduction}
	We study the quasi-order of (topological) embeddability on definable functions between Polish spaces. A function $f$ is said to embed into a function $g$, in symbols $f\emb g$, if there exist topological embeddings $\sigma:\dom f\to \dom g$ and $\tau:\im f \to \im g$ such that $g\circ \sigma= \tau \circ f$. Since the identity is an embedding and embeddings compose, embeddability is a reflexive and transitive relation on functions, namely a quasi-order.
	
	This way of comparing functions was first used by \textcite{soleckidecomposing} who proved the existence of a non $\sigma$-continuous Baire class $1$ function which is minimal for embeddability.       
	Several basis results of this sort have later been obtained \cite{pawlikowski2012decomposing,carroymiller1}, therefore confirming the relevance of this quasi-order for the study of definable functions.
	
	In the first part of this article, we focus specifically on the properties of the quasi-order of embeddability when restricted to the space $\C(X,Y)$ of continuous functions between two Polish \zerodim{} spaces $X$ and $Y$. We endow these spaces of functions with the topology of compact convergence, which coincides in this case with the compact-open topology. While this turns $\C(X,Y)$ into a Polish space if and only if $X$ is locally compact, this choice is briefly discussed in \cref{subsec:limitationTopologies}.
	
	Our first result shows that the embeddability on $\C(X,Y)$ is quite complicated in most cases, in particular when both $X$ and $Y$ are uncountable. More precisely, let $\CCG$ denote the Polish space of graphs on $\N$ and consider the quasi-order of injective homomorphisms given by $G\ihom H$ iff there exists an injective homomorphism from $G$ to $H$.
	\Textcite[Thm 3.5]{louveau2005complete} showed that the homomorphism quasi-order on countable directed graphs is analytic complete and their proof also shows that $(\CCG,\ihom)$ is an analytic complete quasi-order. We prove the following as \cref{t:thechaossideYY}.
	
	\begin{theorem}\label{thmINTRO1}
		Let $X$ and $Y$ be Polish \zerodim{} spaces. If $X$ has infinitely many non-isolated points and $Y$ is not discrete, then there exists a continuous map $\G \to \C(X,Y)$ which reduces $\ihom$ to $\emb$.
	\end{theorem}
	
	Next we study the case where $\C(X,Y)$ does not fall into the scope of the above theorem. Recall that a quasi-order is a well-quasi-order if it is well-founded and has no infinite antichain. Of course an analytic complete quasi-order does not reduce to a well-quasi-order. In fact, while analytic complete quasi-orders are quite wild, the well-quasi-orders are very tame and provide a robust notion of hierarchy. In \cref{sec:ORDER}, we prove the following result as \cref{t:locallyconstant,t:finitelymanylimitpoints}:
	
	\begin{theorem}\label{thmINTRO2}Let $X$ and $Y$ be Polish \zerodim{} spaces. If either $X$ is locally compact with finitely many non-isolated points, or if $Y$ is discrete, then embeddability is a well-quasi-order on $\C(X,Y)$.
	\end{theorem}
	
	Our methods unfortunately do not allow us to drop the assumption that $X$ is locally compact in \cref{thmINTRO2} (see \cref{section:last}).
	
	However combining our two theorems, we obtain a dichotomy for the spaces of functions $\C(X,Y)$ with $X$ locally compact. In particular when $X$ is compact, then $\C(X,Y)$ is Polish and the topology coincides with that of the uniform convergence. We show that in this case embeddability is an analytic quasi-order (\cref{EmbedIsAnalytic}). We therefore obtain the following dichotomy.

	\begin{corollary}\label{c:intro}
		Let $X$ and $Y$ be Polish \zerodim{} with $X$ compact. Then exactly one of the following holds: either the embeddability on $\C(X,Y)$ is analytic complete, or it is a well-quasi-order.
	\end{corollary}

	In \cref{sec:RankEmbed}, we move our attention to Borel functions on some
	Polish space and investigate the existence of maximal elements in Baire classes. Using the universality property of the Hilbert cube, we notice that the projection $p:[0,1]^\omega\times[0,1]^\omega \to [0,1]^\omega$ is a maximum with respect to embeddability on continuous functions between Polish spaces. We then proceed to show that the situation is completely different for the other Baire classes (\cref{t:rank}).
	
	\begin{theorem}\label{thmINTRO3}
		Let $X,Y$ be Polish spaces such that $X$ is uncountable and \zerodim{} and $|Y| \geq 2$. For every countable ordinal $\xi\geq1$, there is no maximal element for embeddability among the Baire class $\xi$ functions from $X$ to $Y$.
	\end{theorem}

	\section{Preliminaries}
	
	\subsection{Spaces}
	
	Unless otherwise specified, all topological spaces are assumed to be Polish \zerodim{}. We use variable symbols $X,Y,Z$ for spaces. Our general reference for the descriptive set theory of these spaces is \cite{kechris2012classical}. For clarity we write $(x_n)_n$ to denote a sequence and omit the index set which is assumed to be $\N$. We write $x_n\to x$ in $X$ to say that the sequence $(x_n)_n$ converges to $x$ in a space $X$ and we just write $x_n\to x$ when the space is clear from the context. Since all our spaces are metrizable, a function $\sigma:X\to Y$ is a (topological) embedding if and only if for every $x\in X$ and every sequence $(x_n)_n$ in $X$ the following equivalence holds:
	\[
	\text{$x_n\to x$ in $X$ if and only if $\sigma(x_n)\to \sigma(x)$ in $Y$.}
	\]

	The notation $X=\bigsqcup_{i\in\N}X_i$ stands for the topological sum of a sequence $(X_i)_{i\in\N}$ of spaces and we take $\bigcup_{i\in\N} \{i\}\times X_i$ to be the underlying set.
	This operation is functorial and if we have continuous functions $f_i:X_i \to Y$ for all $i\in\N$, then $\bigsqcup_{i\in\N}f_i$ denotes the continuous map $\bigsqcup_i X_i \to Y$ given by $(i,x)\mapsto f_i(x)$.
	
	If $X$ is a locally compact metrizable space that is not compact, the Alexandrov compactification, or one-point compactification, of $X$ is obtained by adding a new point $\infty$ to $X$ and consider the topology given by the sets $U\subseteq X$ open in $X$ and the sets $V=\{\infty\} \cup X\setminus C$ where $C$ is compact in $X$.
	
	We view countable ordinals as topological spaces for the order topology. For example, $\omega$ is homeomorphic to $\N$ with the discrete topology and $\omega+1$ is homeomorphic to the one-point compactification of $\N$.
	
	The ordinal $\omega^2$ is homeomorphic to the sum of countably many copies of $\omega+1$. Finally, the ordinal $\omega^2+1$ is homeomorphic to the one-point compactification of $\omega^2$. Throughout the article we freely identify $\omega^2$ with $\bigsqcup_{n\in \N} \omega+1$ and $\omega^2+1$ with the Alexandrov compactification of $\bigsqcup_{n\in \N} \omega+1$.
	
	We mention here for later use an easy lemma which allows us to glue together embeddings.\footnote{\cref{LemWedgeEmb} really says that the wedge sum of two embeddings between pointed metrizable spaces is again an embedding.
	}

	\begin{lemma}\label{LemWedgeEmb}
		Let $X$ be a Polish space.
		Suppose that $F_0$ and $F_1$ are two closed sets of $X$ with $X=F_0\cup F_1$ and $F_0\cap F_1=\{z\}$ for some $z\in X$. If $\sigma_i:F_i\to F_i$, $i=0,1$, are two embeddings such that $\sigma_0(z)=\sigma_1(z)$, then $\sigma=\sigma_0\cup \sigma_1:X\to Y$ is an embedding.
	\end{lemma}
	\begin{proof}
		Let $(x_n)_n$ be any sequence in $X$ and $x\in X$. Let us show that $x_n\to x$ iff $\sigma(x_n)\to \sigma(x)$. In case $(x_n)_n$ is eventually in $F_0$ or $F_1$, this follows from the fact that $\sigma_i$ is an embedding for $i=0,1$. Otherwise we can partition $(x_n)_n$ into two sequences $(x^0_n)_n\subseteq F_0$ and $(x^1_n)_n\subseteq F_1$. Then $x_n\to x$ iff $x^i_n\to x$ for $i=0,1$. So in particular we must have $x=z$. Hence this is in turn equivalent to $\sigma_i(x^i_n)\to \sigma(z)$ for $i=0,1$. Which is finally equivalent to $\sigma(x_n)\to \sigma(z)$.
	\end{proof}
	
	\subsection{Quasi-orders}
	
	A \emph{quasi-order} is a transitive reflexive relation $\leq_Q$ on some set $Q$. As it is customary, we henceforth make an abuse of terminology and refer to the pair $(Q,\leq_Q)$ simply as $Q$ when there is no danger of confusion.
	
	We say that a quasi-order $Q$ is \emph{well-founded} if every non-empty subset of $Q$ admits a (not necessarily unique) $\leq_Q$-minimal element. A \emph{well-quasi-order}, or \wqo{}, is a well-founded quasi-order that has no infinite \emph{antichain}, that is no infinite subset of pairwise incomparable elements.
	
	For clarity, we restrict our use of the terms \emph{embed} and \emph{embedding} to topological spaces or functions. Instead of referring to embeddings when dealing with quasi-orders, we talk about reductions. Given two quasi-orders $P$ and $Q$ respectively, a \emph{reduction} from $(P,\leq_P)$ to $(Q,\leq_Q)$ is a map $\varphi:Q\to P$ such that $p\leq_P p'\lrao\varphi(p)\leq_Q\varphi(p')$ for all $p,p'$ in $P$. We say that $P$ \emph{reduces} to $Q$ if there exists a reduction from $P$ to $Q$. When $P$ and $Q$ are topological spaces we say that $(P,\leq_P)$ \emph{continuously} (resp. \emph{Borel}) \emph{reduces} to $(Q,\leq_Q)$ when a continuous (resp. Borel) reduction exists.
	
	A subset $A$ of a Polish space $X$ is analytic (or $\ana$) if it is the projection of a closed subset of the product $\omega^\omega\times X$ of $X$ with the Baire space. A quasi-order $\leq$ on some Polish space $X$ is \emph{analytic} if it is analytic as a subset of $X\times X$. It is \emph{analytic complete} if it is analytic and moreover every analytic quasi-order on some Polish space Borel reduces to it.
	
	Notice that since the quasi-order of equality on $\N$ does not reduce to any \wqo{}, a \wqo{} is never analytic complete.
	
	We make use of the following example of analytic complete quasi-order in \cref{sec:CHAOS}.
	It is defined on (simple) graphs with vertex set $\N$. These can be viewed as pairs $G=(\N,E)$ where the edge relation $E$ is a subset of $[\N]^2$, the set of all pairs of natural numbers. Identifying any graph $G=(\N,E)$ with its edge relation $E$, we view $G$ as an element of the Polish space $2^{[\N]^2}$. This makes the set of graphs on $\N$ a Polish space and we denote it by $\mathbb{G}$.
	
	An \emph{injective homomorphism} from $G_{0}=(\N,E_{0})$ to $G_{1}=(\N,E_{1})$  is an injective map $h:\N\to \N$ such that $\{m,n\}\in E_{0}$ implies $\{h(m),h(n)\}\in E_{1}$ for all pairs $\{m,n\}$. We write $G_{0}\ihom G_{1}$ when there exists an injective homomorphism from $G_{0}$ to $G_{1}$. The proof of \cite[Thm 3.5]{louveau2005complete} also gives the following result on which \cref{sec:CHAOS} relies.
	
	\begin{theorem}[Louveau--Rosendal]\label{LouveauRosendal}
		The quasi-order $(\CCG,\ihom)$ is analytic complete.
	\end{theorem}

	\subsection{Some Cantor--Bendixson analysis}
	
	We make a simple yet quite essential use of the Cantor--Bendixson analysis throughout the article. We briefly recall here what we need. We say that a point $x$ in a space $X$ is \emph{isolated} if $\{x\}$ is open in $X$, otherwise we say that $x$ is a \emph{limit}. We let $\isolated(X)$ denote the set of isolated points of the space $X$.
	
	The \emph{Cantor--Bendixson sequence of derivatives} of a space $X$ is defined inductively by:
	\begin{itemize}
		\item $\CB_0(X)=X$,
		\item $\CB_{\alpha+1}(X)=\CB_\alpha(X)\setminus\isolated(\CB_\alpha(X))$,
		\item $\CB_\alpha(X)=\bigcap_{\beta<\alpha}\CB_\beta(X)$ in case $\alpha$ is limit.
	\end{itemize}
	
	The \emph{Cantor--Bendixson rank} of a space $X$ is the smallest ordinal $\alpha$ such that $\CB_\alpha(X)=\CB_{\alpha+1}(X)$, we denote it by $\CB(X)$. Since we are only dealing with second countable spaces, this rank will always be countable. Spaces of rank $0$ are either empty or perfect. The case of countable spaces is of particular interest for us; in that case the sequence of Cantor--Bendixson derivatives is eventually empty, and the last non-empty derivative, if there is one, is always discrete.
	
	The following results are folklore.
		\begin{fact}\label{fact:onCBrank} We say that a space $X$ is simple if $\CB(X)=\alpha+1$ for some countable ordinal $\alpha$, and $\CB_\alpha(X)$ is a singleton.
			\begin{itemize}
				\item If $\sigma:X\to Y$ is an embedding between countable spaces, then $\sigma(\CB_\alpha(X))\sub \CB_\alpha(Y)$ for all countable $\alpha$.
				\item Every countable Polish space is a topological sum of countably many simple spaces.
				\item If $X$ is compact and countable, then $\CB(X)=\alpha+1$ for some countable $\alpha$, the cardinality of $\CB_\alpha(X)$ is finite and $X$ is the topological sum of $|\CB_\alpha(X)|$ simple spaces of rank $\alpha+1$.
			\end{itemize}
		\end{fact}
		For the proof, in a slightly more general context, see for instance \cite[Proposition 2.2 and Lemma 2.4]{carroy2}.

		The following classification of Polish spaces with exactly one limit point\footnote{Or equivalently, separable metrizable spaces with exactly one limit point.} is used in \cref{sec:CHAOS}. Here we denote by $[\omega]^2_+$ the set $[\omega]^2\cup\{\emptyset\}$ viewed as a subspace of the Cantor space.
		
		\begin{proposition}\label{prop:simpleCB2}
			Up to homeomorphism, there are only $3$ Polish spaces with a single limit point, namely: 
			\[\text{$\omega+1$, $\omega\sqcup (\omega+1)$ and $[\omega]^2_+$}.
			\]
		\end{proposition}
		Notice that $\omega+1$ is compact and $\omega\sqcup (\omega+1)$ is locally compact, while $[\omega]^2_+$ is not locally compact.
		\begin{proof}
			Let $X$ be Polish with single limit point $x_\infty$ and fix a compatible metric for $X$.
			Take a sequence $(U_n)_{n\in \N}$ of open balls centered at $x_\infty$ with vanishing diameter. Notice that for all $n$, $X\setminus U_n$ is a countable set of isolated points of $X$, hence it is open. Therefore each $U_n$ is actually clopen in $X$ and so are the sets $V_n=U_n\setminus U_{n+1}$. We distinguish three cases.
			
			Suppose first that $V_n$ is finite for all $n\in\N$ and so $X$ is compact. Then any bijective map $h:X\to \omega+1$ which maps $x_\infty$ to $\omega$ is a homeomorphism.
			
			Next assume that there is an $N\in\N$ such that $V_n$ is finite for all $n\geq N$ while $X\setminus U_N$ is infinite. Now $X\setminus U_N$ is clopen in $X$ and as an infinite discrete space it is homeomorphic to $\omega$. Since $U_N$ falls in the first case, $X$ is homeomorphic to $\omega\sqcup(\omega+1)$.
			
			Otherwise, by possibly going to a subsequence of neighborhoods $(U_n)_{n\in\N}$, we can assume that $V_n$ is infinite for all $n$. Fix an enumeration $(x_{m,n})_{n\in\N}$ of $V_m$ for all $m$, and define $h:X\to [\omega]^2_+$ by $\tau(x_{m,n})=\{m,m+n\}$ and $h(x_\infty)=\emptyset$. One easily checks that $h$ is a homeomorphism, as desired. 
		\end{proof}
		
		\section{Embedding functions}\label{sec:EMbedding}
		
\subsection{Let's embed them!}		
		
		The following ordering on functions is the main object of study of this article.
		\begin{definition}
			Let $f:X\rao Y$ and $f':X'\rao Y'$ be two functions between topological spaces. We say that $f$ \emph{embeds} in $g$, in symbols $f\emb g$, if there exist embeddings $\sigma:X\rao X'$ and $\tau:\im f\rao \im g$ such that $\tau\circ f=f'\circ\sigma$. In this case we also say that $(\sigma,\tau)$ \emph{embeds} $f$ in $g$.
		\end{definition}

		We start by making a simple but important observation on this definition. If $(\sigma,\tau)$ embeds $f$ into $g$, then the embedding $\tau:\im f \to \im g$ is entirely determined by $\sigma$, $f$ and $g$ through the equation $\tau(f(x))=g(\sigma(x))$ for all $x\in X$. In particular, if $(\sigma,\tau)$ and $(\sigma,\tau')$ both embed $f$ into $g$, then $\tau=\tau'$. With this in mind, we sometimes just say that $\sigma$ \emph{embeds} $f$ into $g$ if $\sigma$ is an embedding and there exists a (unique) embedding $\tau:\im f \to \im g$ such that $\tau\circ f=g\circ \sigma$.
		
		Observe that in order for an embedding $\sigma:X\to X'$ to embed $f$ into $g$ it is necessary that $\sigma$ reduces the equivalence relation induced by $f$ to that induced by $g$, namely that
		\[
		\forall x_0,x_1 \in X \quad \big( f(x_0)=f(x_1) \leftrightarrow g(\sigma(x_0))=g(\sigma(x_1)) \big).
		\]
		In which case, the assignment $f(x)\mapsto g(\sigma(x))$ actually defines an injective function $\tau:\im f\to \im g$. Hence for $\sigma$ to embed $f$ into $g$ it is necessary and sufficient that moreover this $\tau$ is an embedding.

		Now given Polish spaces $X$ and $Y$, we endow the space $\C(X,Y)$ of continuous functions from $X$ to $Y$ with the topology of compact convergence. This topology coincides in this case with the compact-open topology (see \cite[Theorem 46.8, p. 285]{munkres2000topology}) which is defined by taking as a subbase sets of the form
		\[
		S_{X,Y}(C,U)=\{f:X\to Y\mid f(C)\sub U\}
		\]
		for some compact subset $C$ of $X$ and some open subset $U$ of $Y$. We recall that if $X$ is Hausdorff and $\mathcal{S}$ is a subbase for the topology of $Y$, then the sets of the form $S_{X,Y}(C,U)$ where $U$ belongs to $\mathcal{S}$ already form a subbase for the compact open topology of $\C(X,Y)$ \cite[Lemma 2.1]{jackson1952spaces}.
		
		When $X$ is locally compact and $Y$ is Polish, then $\C(X,Y)$ is Polish (see \cite[Theorem 1, p.93 and Theorem 3, p. 94]{kuratowski1968topology}), while in the non-locally compact case this topology is not even metrizable. We address the existence of a natural topology on these spaces at the end of this section.

		When moreover $X$ is compact, then the topology on $\C(X,Y)$ coincides with the uniform convergence topology induced by the uniform metric
		\[
		d_u(f,g)=\sup_{x\in X} d_Y(f(x),g(x)),
		\]
		where $d_Y$ is some compatible metric for $Y$. In this case, $\C(X,Y)$ is Polish \cite[(4.19)]{kechris2012classical} and we show that the quasi-order of embeddability is analytic. Notice that our general assumption that the spaces are \zerodim{} is actually unnecessary for this result.
		
		\begin{theorem}\label{EmbedIsAnalytic}
			Let $X$ and $Y$ be Polish with $X$ compact. Then embeddability is an analytic quasi-order on $\C(X,Y)$.
		\end{theorem}
		
        \begin{proof}
We show that the relation $E$ on $\C(X,X)\times \C(X,Y)^2$ given by
\begin{align*}
(\sigma,f,g)\in E \IFF \text{$\sigma$ embeds $f$ into $g$}
\end{align*}
is Borel. Since $f\emb g$ iff $\exists \sigma \in \C(X,X) \ (\sigma,f,g)\in E$, this implies that the relation $f\emb g$ is analytic, as desired.

Fix a countable dense subset $D=\{d_n:n  \in \omega\}$ of $X$ and compatible complete metrics $d_X$ and $d_Y$ on $X$ and $Y$, respectively.

First we observe that for every $x,x' \in X$ and $k>0$ the set $D_{x,x'}^k=\{f \in \C(X,Y):d_Y(f(x),f(x'))\leq \frac{1}{k}\}$ is closed in $\C(X,Y)$. To see this let $f\in \C(X,Y)$ with $d_Y(f(x),f(x'))\geq \frac{1}{k}+\epsilon$ for some $\epsilon>0$. Then for every $g\in \C(X,Y)$ with $d_u(f,g)<\frac{\epsilon}{2}$ we have $d_Y(g(x),g(x'))\geq d_Y(f(x),f(x'))- 2\cdot d_u(f,g)> \frac{1}{k}$. Hence the complement of $D_{x,x'}^k$ is open, as desired. 

Next we prove that for all $\sigma \in \C(X,X)$, $\sigma$ is an embedding if and only if
\[
 \forall k>0 \ \exists l>0 \ \forall m,n \ \bigl( d_Y(\sigma(d_n),\sigma(d_m))\leq \frac{1}{l} \text{ implies } d_X(d_n,d_m)\leq \frac{1}{k} \bigr),
\]
which shows that the set of embeddings is Borel in $\C(X,X)$. Suppose that $\sigma\in \C(X,X)$ is an embedding and therefore admits a continuous inverse $\sigma^{-1}: \im \sigma \to X$. Since $X$ is compact, so is $\im \sigma$, hence $\sigma^{-1}$ is uniformly continuous. This shows the forward implication. Conversely, assume that $\sigma\in \C(X,X)$ satisfies the above formula. By compactness of $X$, it suffices to show that $\sigma$ is injective. To this end assume that $\sigma(x)=\sigma(x')$ and pick two sequences in $D$ with $b_n\to x$ and $c_m \to x'$. Since both $(\sigma(b_n))_n$ and $(\sigma(c_m))_m$ converge to $\sigma(x)$ it follows that the sequence $(\sigma(b_0),\sigma(c_0),\sigma(b_1),\sigma(c_1),\ldots)$ is Cauchy. As $\sigma$ satisfies above formula, it follows that $(d_0,c_0,d_1,c_1,\ldots)$ is Cauchy too, which implies that $x=x'$.

Finally we claim that $(\sigma,f,g)\in E$ iff $\sigma$ is an embedding,
\begin{align*}
\forall k>0 \ \exists l>0 \ \forall n,m \ \bigl( d_Y(f(d_n),f(d_m))\leq \frac{1}{l} \to d_Y(g\circ\sigma(d_n),g\circ\sigma(d_m))\leq \frac{1}{k}\bigr)
\intertext{and}
\forall k>0 \ \exists l>0 \ \forall n,m \ \bigl(d_Y(g\circ\sigma(d_n),g\circ\sigma(d_m))\leq \frac{1}{l} \to d_Y(f(d_n),f(d_m))\leq \frac{1}{k}\bigr).
\end{align*}
This shows that $E$ is Borel. Suppose first that $(\sigma,f,g)\in E$. Then $\sigma$ is an embedding and we have a unique embedding $\tau:\im f \to \im g\circ \sigma$ such that $(\sigma,\tau)$ embeds $f$ into $g$. 
Since $\im f$ and $\im g\circ \sigma$ are compact, both $\tau$ and its inverse are uniformly continuous. Moreover as $\tau(f(x))=g\circ \sigma(x)$ for all $x\in X$, it follows that the two above formulas are satisfied. 

Conversely suppose that $\sigma:X\to X$ is an embedding such that the two above formulas hold with respect to some $f$ and $g$ in $\C(X,Y)$. Notice that in particular, we have $f(d_n)=f(d_m)$ iff $g\circ \sigma(d_n)=g\circ \sigma(d_m)$ for every $m,n\in \omega$. We therefore have a well defined bijective map $\tilde{\tau}:f(D)\to g\circ \sigma(D)$ given by $\tilde{\tau}(f(d))=g\circ\sigma(d)$.  Since $f(D)$ is dense in $\im f$ and $\tilde{\tau}$ is uniformly continuous, it extends to a continuous map $\tau:\im f \to \im g\circ \sigma$. Moreover as before this continuous extension is injective and so $\tau$ is an embedding.  Finally, since $\tau\circ f$ and $g\circ \sigma$ are both continuous and coincide on $D$, they are equal on $X$. Therefore $\tau$ is an embedding such that $(\sigma,\tau)$ embeds $f$ in $g$, as desired.
\end{proof}

		\begin{remark}
			When $X$ and $Y$ are Polish and $X$ is locally compact, the space $\C(X,Y)$ is Polish. One might therefore be tempted to think that the embeddability relation is analytic for such spaces as well. However, if we only assume that $X$ is locally compact in \Cref{EmbedIsAnalytic} then the conclusion does not hold. There exist a locally compact Polish space $X$ for which embeddability is not analytic on $\C(X,X)$ (see \cref{section:last}).
		\end{remark}
		
		We end this subsection by collecting a few basic facts for later use.
		Observe that any embedding $j:Y\to Z$ induces a map $j^\circ:\C(X,Y)\to \C(X,Z)$ given by $j^\circ(f)=j\circ f$. It should come as no surprise that this map actually is a continuous reduction with with respect to embeddability on functions. 
		
		\begin{proposition}\label{cor:embeddingrangesideY}
			For all spaces $X,Y,Z$, if $Y$ embeds in $Z$ then $(\C(X,Y),\emb)$ reduces continuously to $(\C(X,Z),\emb)$.
		\end{proposition}
		\begin{proof}
			Let $j: Y\to Z$ be an embedding. Then the induced map $j^\circ:\C(X,Y) \to \C(X,Z)$, given by $j^\circ(f)=j\circ f$ is continuous by \cite[Theorems 2.2.4.b, p. 18]{mccoy2006topological}.
			Now for any embedding $\sigma:X\to X$ we show that $\sigma$ embeds $f$ into $g$ iff $\sigma$ embeds $j^\circ(f)$ into $j^\circ(g)$. Since $j$ is an embedding, we have
			\[
			\forall x_0,x_1\in X \quad f(x_0)=f(x_1) \IFF g(\sigma(x_0))=g(\sigma(x_1))
			\]
			if and only if
			\[
			\forall x_0,x_1\in X \quad j^\circ( f)(x_0)=j^\circ(f)(x_1) \IFF j^\circ(g)(\sigma(x_0))=j^\circ(g)(\sigma(x_1)).
			\]
			Therefore an embedding $\sigma:X\to X$ induces a injective function $\tau:\im f \to \im g$ such that $\tau\circ f=g\circ \sigma$ iff it induces an injective function $\tau':\im j^\circ(f)\to \im j^\circ(g)$ with $\tau'\circ j^\circ(f)=j^\circ(g)\circ \sigma$. Now notice that $\tau$ is an embedding just in case for all sequences $(x_n)_n$ in $X$ and all $x\in X$, $f(x_n)\to f(x)$ iff $g(\sigma(x_n))\to g(\sigma(x))$.
			Since $j$ is an embedding, we see that $\tau$ is an embedding exactly when so is $\tau'$. Therefore $f\emb g$ iff $j^\circ(f)\emb j^\circ(g)$ and so $j^\circ$ is a reduction as desired.
		\end{proof}

		In \cref{sec:CHAOS}, we will make use of functions which are almost embeddings while they lack injectivity.
		Whenever a function $\rho:X\to Y$ is not injective, we can pick distinct points $x$ and $x'$ in $X$ with $\rho(x)=\rho(x')$ and therefore the sequence $(x,x',x,x',\cdots)$ does not converge while $(\rho(x),\rho(x'),\rho(x),\rho(x'),\cdots)$ is actually constant. 
		With this very simple observation in mind, we make the following definition.
		
		\begin{definition}
			Let $\rho: X\to X'$ be a function. We say that a sequence $(x_n)_n$ in $X$ is \emph{$\rho$-trivial} if the corresponding sequence $\big(f(x_n)\big)_n$ of images by $f$ is ultimately constant. We say that $\rho$ is a \emph{pseudo-embedding} if $\rho$ is continuous and 
			and for all not $\rho$-trivial sequence $(x_n)_n$ and all $x\in X$ we have
			\[
			x_n\to x \quad\text{if and only if}\quad \rho(x_n)\to\rho(x).
			\]
		\end{definition}
		
		Notice that an injective pseudo-embedding is just an embedding. One reason for considering this weakening of the notion of embedding is that the relation of embeddability simplifies a little when restricted to such functions.
		
		\begin{proposition}\label{EmbSigmaQuasiEmb}
			Let $f:X\to Y$ and $g:X'\to Y'$ be two pseudo-embeddings. Then $f\emb g$ iff there exists an embedding $\sigma:X\to X'$ such that $f(x)=f(x')$ iff $g(\sigma(x))=g(\sigma(x'))$ for all $x,x'\in X$.
		\end{proposition}
		\begin{proof}
			Clearly this is necessary, so let us prove this actually is sufficient. Let $\sigma:X\to X$ be an embedding such that $f(x)=f(x')$ iff $g(\sigma(x))=g(\sigma(x'))$ for all $x,x'\in X$. This ensures that the assignment $f(x)\mapsto g(\sigma(x))$ defines an injective map $\tau:\im f \to \im g$. To see that this $\tau$ is an embedding, take some sequence $(x_n)_n$ in $X$ and let $x\in X$ with the aim to show that
			\[
			\text{$f(x_n)\to f(x)$ in $Y$} \IFF \text{$g(\sigma(x_n))\to g(\sigma(x))$ in $X$}.
			\]
			Notice $(x_n)_n$ is $f$-trivial iff $(\sigma(x_n))_n$ is $g$-trivial.
			If $(x_n)_n$ is $f$-trivial and so $(\sigma(x_n))_n$ is $g$-trivial too, then the sequences on both sides are eventually constant and we are done.
			So suppose that $(x_n)_n$ is not $f$-trivial and so $(\sigma(x_n))_n$ is not $g$-trivial either. Since $f$ and $g$ are pseudo-embeddings, we see that each side of the equivalence hold iff $x_n\to x$, which finishes the proof.
		\end{proof}
		
		If $\rho:X\to Y$ is a pseudo-embedding and $\tau:Y\to Y$ is an embedding, then there is an easy sufficient (and in fact necessary) condition for the existence of an embedding $\sigma:X\to X$ such that $(\sigma,\tau)$  is a witness for $\rho\emb \rho$.
		
		\begin{proposition}\label{prop:PseudoEmbLiftingEmb}
			Let $\rho:X\to Y$ be a pseudo-embedding. If $\tau:Y\to Y$ is an embedding such that $\rho^{-1}(y)$ embeds into $\rho^{-1}(\tau(y))$ for every $y\in Y$, then there exists an embedding $\sigma:X\to X$ such that $\rho\circ \sigma=\tau\circ \rho$.
		\end{proposition}
		\begin{proof}
			Pick for every $y\in \im \rho$ an embedding $\sigma_y:\rho^{-1}(y)\to \rho^{-1}(\tau(y))$. We show that $\sigma=\bigcup_{y\in Y}\sigma_y$ is an embedding such that $(\sigma,\tau)$ embeds $\rho$ in itself. Clearly we have $\rho\circ \sigma=\tau\circ \rho$. To see that $\sigma$ is an embedding, let $(x_n)_n$ be a sequence in $X$ and $x\in X$. We claim that $x_n\to x$ iff $\sigma(x_n)\to\sigma(x)$. Observe that $(x_n)_n$ is $\rho$-trivial exactly when $(\sigma(x_n))_n$ is. If $(x_n)_n$ is  $\rho$-trivial with $\big(\rho(x_n)\big)_n$ eventually equal to $y$, then the claim follows from the fact that $\sigma_y$ is an embedding. Otherwise, since $\rho$ is a pseudo-embedding $x_n\to x$ iff $\rho(x_n)\to \rho(x)$, so as $\tau$ is an embedding this is equivalent to $\tau(\rho(x_n))\to \tau(\rho(x))$. This is in turn equivalent to $\rho\circ \sigma(x_n)\to\rho\circ \sigma(x)$, hence using that $\rho$ is a pseudo-embedding once again we get that this is finally equivalent to $\sigma(x_n)\to\sigma(x)$.
		\end{proof}
		
		\subsection{A limitation result on possible topologies for continuous functions}\label{subsec:limitationTopologies}
		
		A natural question arises when we consider spaces of the form $C(X,Y)$: does there exist a nice topology on these spaces? More precisely does there exist a standard Borel structure on $C(X,Y)$ so that the \emph{evaluation map} $C(X,Y) \times X \to Y$ defined by $(f,x) \mapsto f(x)$ is Borel?
		
		We first observe that the answer is affirmative when $X$ is $\sigma$-compact. Let $X=\bigcup_{n \in \omega} K_n$ with compact sets $K_n$ and identify $C(X,Y)$ with a subset of the Polish space $\prod_{n \in \omega} C(K_n,Y)$ via the injection $f\mapsto(f\rest{K_n})_n$. It is enough to notice the following fact.
        
\begin{proposition}
The space $\C(X,Y)$ is Borel in $\prod_{n \in \omega} C(K_n,Y)$.
\end{proposition}
        
\begin{proof}
Clearly $\C(X,Y)$ is a coanalytic subset of $\prod_{n \in \omega} C(K_n,Y)$ and we show it is also analytic, hence Borel. Fix metrics $d_X$ ad $d_Y$ on $X$ and $Y$, and countable dense subsets $D_i\subseteq K_i$. We claim that for every $(f_i)_i\in \prod_{n \in \omega} C(K_n,Y)$, $\bigcup_{i\in \N} f_i$ is the graph of a continuous function if and only 
if for every $n>0$ there is a sequence $(\alpha_i)_i$ with $\alpha_i>0$ such that for every $i,j\in \N$ if $d\in D_i$, $d' \in D_j$ and $d_X(d,d')<\frac{1}{\alpha_i}$ then $d_Y(f_i(d),f_j(d'))<\frac{1}{n}$. This definition is analytic, as desired.

Suppose that $\bigcup_i f_i$ is not the graph of a function, so that for some $x$ there is $i,j\in \N$ and $N>0$ such that $x\in K_i\cap K_j$ and $d_Y(f_i(x), f_j(x))>\frac{1}{N}$ for some $N$. Towards a contradiction, assume that $(f_i)_i$ satisfies the above definition for $n=3N$ and let $(\alpha_i)_i$ witness this fact. By continuity of $f_i$ on $K_i$ and density of $D_i$, there exists $d\in D_i$ with $d_X(x,d)<\frac{1}{2\cdot\alpha_i}$ and $d_Y(f_i(x),f_i(d))<\frac{1}{3N}$. Similarly, there is $d'\in D_j$ with $d_X(x,d')<\frac{1}{2\cdot\alpha_i}$ and $d_Y(f_j(x),f_j(d'))<\frac{1}{3N}$. In particular, $d_X(d,d')<\frac{1}{\alpha_i}$, so $d_Y(f_i(d),f_j(d))<\frac{1}{3N}$. It follows that $d_Y(f_i(x),f_j(x))\leq \frac{1}{N}$, a contradiction. Hence if $(f_i)_i$ satisfies our condition, then $\bigcup_i f_i$ is the graph of a function from $X$ to $Y$. The argument to show that moreover this function is continuous is similar.


Conversely, assume that $f:X\to Y$ is continuous and fix $n>0$. By continuity of $f$, for each $x \in X$ there exists a $\delta_x$ so that $\diam(f(B(x,\delta_x)))<\frac{1}{n}$. Fix $i\in\N$ and notice that $K_i\subseteq \bigcup_{x \in K_i} B(x,\frac{1}{2}\cdot\delta_x)$. By compactness, $K_i$ is already covered by finitely many open balls centered at some $(x_k)_{k<m}$ and we take $\alpha_i>0$ such that $\frac{1}{\alpha_i} < \frac{1}{2}\cdot\min \{\delta_{x_k}\mid k<m\}$. Then if $x \in K_i$, then for some $k<m$ we have $x\in B(x_k,\frac{1}{2}\cdot\delta_{x_k})$ so in particular $B(x,\frac{1}{\alpha_i}) \subseteq B(x_k, \delta_{x_k} )$ which implies $\diam(f(B(x,\frac{1}{\alpha_i})))<\frac{1}{n}$.
\end{proof}

Therefore when $X$ is a $\sigma$-compact metric space and $Y$ is Polish, there is a standard Borel structure on $\C(X,Y)$ which makes the evaluation map Borel. We next show that this fails if $X$ is not $\sigma$-compact and \zerodim{}. In particular, there is no standard Borel structure on $\C(\omega^\omega,\omega^\omega)$ which makes the evaluation map Borel.
As usual, if $B \subseteq X \times Y$, $x\in X$ and $y \in Y$ we denote the horizontal and vertical sections, that is, the sets $\{x' \in X:(x',y)\in B\}$ and $\{y' \in Y:(x,y')\in B\}$, by $B^y$ and $B_x$, respectively.

		\begin{theorem}
			Let $X,Y$ be \zerodim{} Polish spaces with $X$ non-$\sigma$-compact and $|Y| \geq 2$. Then there is no standard Borel structure on $C(X,Y)$ so that the map $C(X,Y) \times X \to Y$ defined by $(f,x) \mapsto f(x)$ is Borel.
		\end{theorem}
		
		\begin{proof}
			
			Suppose towards a contradiction that there is such a structure. First we reduce the problem to $X=\omega^\omega$. By the isomorphism of standard Borel spaces there exists a Borel bijection $b:\omega^\omega \to C(X,Y)$. As $X$ is not $\sigma$-compact, it contains a closed subspace $X'$ homeomorphic to $\omega^\omega$, so let $e: \omega^\omega \to X'$ be such a homeomorphism. We define a subset $B$ of $\omega^\omega \times \omega^\omega \times Y$ by $(r,s,y) \in B \iff b(r)(e(s))=y$. Then by our assumption, $B$ is a Borel set and of course, for every $r \in \omega^\omega$ we have that $B_r$ is the graph of a continuous function from $\omega^\omega$ to $Y$. Now as $X'$ is a closed subset of $X$, every continuous function $X' \to Y$ extends to a continuous function $X \to Y$ (by \cite[Corollary 6c, p. 151]{kuratowski1966topology1} and \cite[(4.17)]{kechris2012classical}). Consequently, for every $f:\omega^\omega \to Y$ continuous there exists an $r$ so that $B_r$ is the graph of $f$.
			
			Therefore it is enough to prove that there is no such Borel set $B \subseteq  \omega^\omega \times \omega^\omega \times Y$. Fix a continuous onto map $t:Y \to 2$, and consider the following set:
\[
B'=\{(r,s) \in (\omega^\omega)^2: (\forall y)((r,s,y) \in B \implies t(y)=0)\}.
\]
We have the following:
			\begin{itemize}
				\item $B'$ is Borel, since $(r,s) \in B' \iff (\exists y)((r,s,y)  \in B \implies t(y)=0)$,
				\item for every $r$ the section $B'_r$ is clopen,
				\item for every clopen subset $C$ of $\omega^\omega$ there exists an $r$ with $B'_r=C$.
			\end{itemize}
			Without loss of generality we can suppose that for every clopen set $C$ the set $\{r \in \omega^\omega:B'_r=C\}$ cannot be covered by the union of countably many compact sets: indeed, instead of $B'$ we can consider the set $B'' \subseteq \omega^\omega \times \omega^\omega \times \omega^\omega$ defined by $(r',r,s) \in B'' \iff (r,s) \in B'$. Then for each clopen $C$ the set ${(r',r):B''_{(r,r')}=C}$ contains a copy of $\omega^\omega$. Identifying $(\omega^\omega)^2$ with $\omega^\omega$ via a standard homeomorphism gives the desired set.
			
			Let us denote the usual topology on $\omega^\omega$ by $\tau_u$. Now, every section $B'_x$ is clopen and it follows from \cite[Theorem 28.7]{kechris2012classical} that there exists a Polish topology $\tau_0 \supseteq \tau_u$ on $\omega^\omega$ so that $B'$ is open in $(\omega^\omega,\tau_0) \times (\omega^\omega, \tau_u)$. Applying this again for $(B')^c$ we can get a \zerodim{}  Polish topology $\tau \supseteq \tau_0$ on $\omega^\omega$ so that $B'$ is clopen in the product space $(\omega^\omega,\tau) \times (\omega^\omega, \tau_u)$. Note that if a subset of $\omega^\omega$ is compact with respect to $\tau$, then it is also compact with respect to $\tau_u$. 
Hence, for every clopen $C$, the set $\{x:B'_x=C\}$ cannot be covered by countably many $\tau$-compact sets. Now, since $(\omega^\omega,\tau)$ is \zerodim{}, there exists a $\sigma$-compact $\tau$-open set $S$, so that $(\omega^\omega \setminus S,\tau)$ is homeomorphic to $(\omega^\omega,\tau)$, so let us denote such a homeomorphism by $h$.
			
			Define $(x,y) \in U \iff (h(x), y) \in B'$. Now,
			$U$ is a clopen subset of $(\omega^\omega)^2$ (indeed, it is the preimage of a clopen set under the map $(h,id)$), for every $C$ clopen there exists an $r$ with $U_r=C$: as $S$ is $\sigma$-compact, for any $C$ there is an $r$ with $B_r=C$ so that $r\not \in S$. This contradicts the fact that there is no universal clopen set.
		\end{proof}
		
		\section{Chaos}
		\label{sec:CHAOS}
		Our starting point is a general and simple observation about the embeddability quasi-order on (not necessarily continuous nor definable) functions between topological spaces.    
		\begin{definition}[The function-to-graph map]
			To every function $f:X\to Y$ between topological spaces we associate a graph $\Gr_{f}=(V_{f},E_{f})$ where $V_{f}=\{x\in X \mid \text{$x$ is a limit point}\}$ and
			\begin{align*}
			\{x,y\}\in E_{f} \IFF& \text{$x\neq y$ and there exist injective sequences}\\
			&\text{$(x_{n})_{n}$ and $(y_{n})_{n}$ in $X$ such that $x_{n}\to x$, $y_{n}\to y$,}\\
			& \text{and $f(x_{n})=f(y_{n})$ for all $n\in\N$.}
			\end{align*}
		\end{definition}
		
		Notice that if $(x_{n})_{n}$ and $(y_{n})_{n}$ witness that $\{x,y\}\in E_{f}$, then so do $(x_{n_{k}})_{k}$ and $(y_{n_{k}})_{k}$ for any strictly increasing sequence $(n_{k})_{k}$ of natural numbers.
		
		The map $f\mapsto \Gr_f$ is actually a homomorphism from $\emb$ to $\ihom$.
		
		\begin{lemma}\label{lemEmbIhom}
			Let $f:X\to Y$ and $f':X'\to Y'$ be two functions between topological spaces. Then $f\emb f'$ implies $\Gr_{f}\ihom \Gr_{f'}$.
		\end{lemma}
		\begin{proof}
			Let $(\sigma,\tau)$ embed $f$ into $f'$. Let us denote by $\check{\sigma}:V_{f}\to V_{f'}$ the restriction of $\sigma:X \to X'$ to the limit points of $X$. This is well defined and injective since $\sigma$ is an embedding, so in particular $\sigma$ must send any limit point of $X$ to some limit point of $X'$. To see that it is a homomorphism from $\Gr_{f}$ to $\Gr_{f'}$, suppose $\{x,y\}\in E_{f}$. There exist injective sequences $(x_{n})_{n}$ and $(y_{n})_{n}$ such that $x_{n}\to x$, $y_{n}\to y$, and $f(x_{n})=f(y_{n})$ for all $n\in \N$. As $\sigma$ is injective and continuous, it follows that $(\sigma(x_{n}))_{n}$ and $(\sigma(y_{n}))_{n}$ are injective sequences converging to $\sigma(x)$ and $\sigma(y)$, respectively. Moreover $f'(\sigma(x_{n}))=\tau(f(x_{n}))=\tau(f(y_{n}))=f'(\sigma(y_{n}))$ for every $n\in \N$, so $\{\sigma(x),\sigma(y)\}\in E_{f'}$ as desired.
		\end{proof}

		In this section we prove that if $Y$ is not discrete and $X$ has infinitely many limit points, then $(\CCG, \ihom)$ reduces continuously to $(\C(X,Y), \emb)$.    Notice that $Y$ is not discrete if and only if $\omega+1$ embeds into $Y$, so by \cref{cor:embeddingrangesideY} it is enough to show that the result holds for $\C(X,\omega+1)$ when $X$ has infinitely many limit points.
		

		\subsection{One special case, or maybe two}
		
		We define a homomorphism $G\mapsto f^G$ from $(\CCG,\ihom)$ to $(\C(\omega^2,\omega+1),\emb)$. We then show that $G$ is isomorphic to $\Gr_{f^G}$ for every $G\in \CCG$. As $f\mapsto \Gr_f$ is a homomorphism by \cref{lemEmbIhom}, it will follow that $G\mapsto f^G$ is actually a reduction, as desired.
		
		Remember that we identify $\omega^2$ with $\bigsqcup_{n\in \N} \omega+1$. The homomorphism is defined as follows: each limit point $(m,\omega)$ corresponds to the vertex $m$ and we pick for every other vertex $n$ a sequence of isolated points converging to $(m,\omega)$. If $\{m,n\}$ is an edge then the sequence converging to $(m,\omega)$ that we picked for $n$ and the sequence converging to $(n,\omega)$ that we picked for $m$ are mapped to the same converging sequence in $\omega+1$. Otherwise, they are mapped to distinct sequences. More formally:

		\begin{definition}[The graph-to-function map]\label{def:GraphToFunction}
			Fix bijections $\bijc{-}_0:\N^2 \to \N$, $\bijc{-}_1:[\N]^2\to\N$ and $\bijc{-}_2:2 \times \N^2\to \N$. For clarity, we write $\bijc{m,n}_1$ instead of $\bijc{\{m,n\}}_1$. Note that if $n\neq m$, then $\bijc{m,n}_0\neq\bijc{n,m}_0$, while $\bijc{m,n}_1=\bijc{n,m}_1$. For every vertex $m\in \N$, define $f^G_m:\omega+1\to\omega+1$ by
			\begin{align*}
			\bijc{n,p}_0&\longmapsto\begin{cases}
            \bijc{0,\bijc{m,n}_0,p}_2\mbox{ if }\{m,n\}\notin E^G,\\
            \bijc{1,\bijc{m,n}_1,p}_2\mbox{ if }\{m,n\}\in E^G,
			\end{cases}\\
			\omega &\longmapsto \omega
			\end{align*}
			We define $f^G:\omega^2\to \omega+1$ to be the sum
			$\bigsqcup_{n\in\N}f^G_n: \bigsqcup_{n\in \N} \omega+1 \to \omega+1$.
		\end{definition}
		
		Clearly each $f^G_n$ is continuous and so $f^G$ is continuous.

		\begin{lemma}\label{lemCtbleGrIhomEmb}
			For every $G,G'\in\CCG$, if $G\ihom G'$ then $f^G\emb f^{G'}$.
		\end{lemma}
		\begin{proof}
			Suppose that $\sigma:\N\to \N$ is an injective homomorphism from $G$ to $G'$. Define $\check{\sigma}:\omega^2 \to \omega^2$ by
			\begin{align*}
			(m,\bijc{n,p}_0)&\longmapsto (\sigma(m),\bijc{\sigma(n),p}_0),\\
			(m,\omega)&\longmapsto (\sigma(m),\omega),
			\end{align*}
			for every $m,n,p \in \N$.
			 Then define
			\begin{align*}
			\tau: \omega+1&\longrightarrow \omega+1\\
			\bijc{i,\bijc{m,n}_i,p}_2 &\longmapsto \bijc{i,\bijc{\sigma(m),\sigma(n)}_i,p}_2,
			\text{ for } i\in\{0,1\}\\
			\omega&\longmapsto \omega.
			\end{align*}
			One easily checks that $(\check{\sigma},\tau)$ embeds $f^{G}$ into $f^{G'}$.
		\end{proof}

		\begin{lemma}\label{lemIsoCtbleGraphs}
			For every graph $G=(\N,E^G)$ in $\CCG$, the map
			\begin{align*}
			j:\N &\longrightarrow V_{f^{G}}
			\\
			n&\longmapsto (n,\omega),
			\end{align*}
			is an isomorphism of graphs from $G$ to $\Gr_{f^{G}}$.
		\end{lemma}
		\begin{proof}
			Notice that $j:V\to V_{f^{G}}$ is bijective. First, take an edge $\{m,n\}\in E^G$. Consider the two sequences given by $x_{p}=(m,\bijc{n,p}_0)$ and $y_{p}=(n,\bijc{m,p}_0)$, $p\in \N$. These sequences are injective, and converge respectively to $j(m)$ and $ j(n)$. Moreover, by the definition of $f^{G}$, we have $f^{G}(x_{p})=f^{G}(y_{p})$ for every $p\in\N$, hence $\{j(m),j(n)\}\in E_{f^{G}}$.
						
			Conversely, suppose that $(x_{p})_{p}$ and $(y_{p})_{p}$ are injective sequences in $\bigsqcup_{n} \omega+1$ witnessing that $\{j(m),j(n)\}\in E_{f^G}$. Since $x_p \to (m,\omega)$ and $y_p \to (n,\omega)$, we can assume by simultaneously passing to subsequences that $(x_{p})_{p}\subseteq \{m\}\times \omega$ and $(y_{p})_{p}\subseteq \{n\}\times \omega$. In particular, $x_0=(m,\bijc{k,p}_0)$ and $y_0=(n,\bijc{k',p'}_0)$ for some $k,k',p,p'\in\N$. But $f^G(x_0)=f^G(y_0)$, so by the definition of $f^G$ this is only possible if $p=p'$, $k'=m$ and $k=n$. Consequently, $\{m,n\} \in E^G$.
		\end{proof}
		
		While the reduction from $(\CCG,\ihom)$ to $\C(\omega^2, \omega+1)$ is simple and illuminating (at least we hope so), we will actually need a reduction to $\C(\omega^2+1,\omega+1)$ in order to derive the general case. 
		
		For every $G\in \CCG$ we extend the function $f^G:\omega^2\to\omega+1$ from \cref{def:GraphToFunction} to a function $\bar{f}^G:\omega^2+1\to \omega+1$ letting $\bar{f}^G(\infty)=\omega$. Using the definition of $f^G$ one easily sees that for every finite set $F\subseteq \omega+1$ of natural numbers, i.e. of isolated points, the set $(\bar{f}^G)^{-1}(F)$ is clopen in $\omega^2+1$ as a finite set of isolated points. This clearly implies that $\bar{f}^G$ is continuous.
		
		The following lemma simplifies the task of proving that both $G\mapsto f^G$ and $G\mapsto \bar{f}^G$ actually yields continuous reductions. 
		
		\begin{lemma}\label{l:RestrictionIsReduction}
			The restriction map $\C(\omega^2+1,\omega+1)\to \C(\omega^2,\omega+1)$, $f\mapsto f\rest{\omega^2}$ is a continuous one-to-one reduction with respect to embeddability.
		\end{lemma}
		\begin{proof}
			To check the continuity, let $K\subseteq \omega^2$ be compact and $U\subseteq \omega+1$ be open. Then $K$ is also compact as a subset of $\omega^2+1$, so for all $f\in \C(\omega^2+1,\omega+1)$:
			\[
			f(K)\subseteq U \IFF f\rest{\omega^2}(K)\subseteq U,
			\]
			which proves continuity.
			
			It is one-to-one since $\omega^2$ is a dense subset of $\omega^2+1$, and if two continuous functions agree on a dense subset then they are equal.
			
			To see that it is a reduction, notice that if $\sigma':\omega^2+1\to\omega^2+1$ is an embedding, then necessarily $\sigma'(\omega^2)=\omega^2$ and so it restricts to an embedding $\sigma:\omega^2\to\omega^2$. Moreover if $\sigma:\omega^2\to \omega^2$ is an embedding then it extends to an embedding $\sigma':\omega^2+1\to\omega^2+1$ by letting $\sigma'(\omega^2)=\omega^2$.
			It follows that for every $f,g\in \C(\omega^2+1,\omega+1)$, $\sigma$ embeds $f$ into $g$ if and only if $\sigma'$ embeds $f\rest{\omega^2}$ into $g\rest{\omega^2}$.
		\end{proof}

		\begin{remark}
			Notice that the restriction map of \cref{l:RestrictionIsReduction} is not a topological embedding of $\C(\omega^2+1,\omega+1)$ into $\C(\omega^2,\omega+1)$. Indeed one can check that the direct image of the open set $S_{\omega^2+1,\omega+1}(\omega^2+1,\omega+1\setminus\{0\})$ is not open in $\C(\omega^2,\omega+1)$.
		\end{remark}

		We now conclude our two special cases, before attacking the general case.
		
		\begin{proposition}\label{p:inaspecialcaseY}
			The maps $G\mapsto f^G$ and $G\mapsto \bar{f}^G$ are continuous reductions from $(\CCG, \ihom)$ to $(\C(\omega^2,\omega+1), \emb)$ and $(\C(\omega^2+1,\omega+1),\emb)$ respectively.
		\end{proposition}

		\begin{proof}
			We claim that it only remains to prove the continuity of $G\mapsto \bar{f}^G$ from $\CCG$ to $\C(\omega^2+1,\omega+1)$ endowed with the compact-open topology in order to conclude. Assuming this fact for now, notice that for every $G\in \CCG$ we have $\bar{f}^G\rest{\omega^2}=f^G$, so by \cref{l:RestrictionIsReduction}, $G\mapsto \bar{f}^G \mapsto \bar{f}^G\rest{\omega^2}=f^G$ is continuous as a composition of continuous maps, and it is a reduction by \cref{lemEmbIhom,lemCtbleGrIhomEmb,lemIsoCtbleGraphs}. But then by \cref{l:RestrictionIsReduction}, $f^G\emb f^H$ iff $\bar{f}^G\emb \bar{f}^H$ for every $G,H\in \CCG$, and so it follows that $G\mapsto \bar{f}^G$ is also a continuous reduction as desired.
			
			Remains to show that $G\mapsto \bar{f}^G$ is continuous. Fix compatible ultrametrics $d$ on $\omega+1$ and $d_G$ on the space of graphs. Since $\omega^2+1$ is compact, the compact-open topology coincides with the uniform convergence topology. So it suffices to check that for every $\varepsilon>0$ there exists a $\delta>0$ so that whenever we have $d_G(G,H)<\delta$ for graphs $G,H$, then for every $x \in \omega^2+1$ we have $d(\bar{f}^G(x),\bar{f}^{H}(x))<\varepsilon$.
			
To this end, fix $\varepsilon>0$ and notice that the set $\{k \in \omega: d(k,\omega) \geq \varepsilon\}$ is finite and that whenever $k$ and $k'$ do not belong to this set then $d(k,k')<\varepsilon$, as $d$ is an ultrametric. Set
\[
F = \{(m,\bijc{n,p}_0)\mid\exists i \ (d(\bijc{i,\bijc{m,n}_i,p}_2,\omega) \geq \varepsilon)\}.
\] 
Then $F$ is a finite set of isolated points in $\omega^2+1$ and $F'=\{ \{m,n\}\in[\omega]^2 \mid \exists p \ (m,\bijc{n,p}_0)\in F\}$ is a finite set of pairs. Now we can find $\delta$ small enough so that $d_G(G,H)<\delta$ ensures that $G$ and $H$ agree on $F'$. By the definition of $\bar{f}^G$ this easily implies $\bar{f}^G|_F=\bar{f}^{H}|_F$. Moreover if $x \not \in F$, then $d(\bar{f}^G(x),\bar{f}^{H}(x))<\varepsilon$. 
		\end{proof}

		\subsection{The general case}
		
		Building on the previous section, we show that if $X$ is Polish \zerodim{} with infinitely many points then $(\CCG,\ihom)$ continuously reduces to $\C(X,\omega+1)$. To do so we will find a continuous map $\pi:X\to \omega^2+1$ such that $G\mapsto \bar{f}^G\circ \pi$ is the desired continuous reduction.
		
		We start by identifying a sufficient condition for a map $\pi:X\to \omega^2+1$ to yield such a reduction. 
		
		\begin{definition}
			Let us say that a pseudo-embedding $\rho:X\to \omega^2+1$ is \emph{regular} if
			\begin{enumerate}
				\item $|\rho^{-1}(\{\infty\})|\leq 1$\label{regularPE1}
				\item for every $m\in\N$, $|\rho^{-1}(\{(m,\omega)\})|=1$, \label{regularPE2}
				\item for every isolated points $y,y'\in\omega^2+1$, $\rho^{-1}(\{y\})$ embeds into $\rho^{-1}(\{y'\})$. \label{regularPE3}
			\end{enumerate}
		\end{definition}

		\begin{lemma} \label{lem:RegularPE}
			Suppose that $\rho:X\to \omega^2+1$ is a regular pseudo-embedding then for every for every $G,H\in \CCG$ we have \[
			G\ihom H \quad \text{if and only if} \quad \bar{f}^G\circ \rho\emb \bar{f}^H\circ \rho.
			\]
		\end{lemma}
		\begin{proof}
			Let us start by remarking that for every $G\in \CCG$ the definition of the function $\bar{f}^G$ is so that: $y$ is limit in $\omega^2+1$ iff $\bar{f}^G(y)=\omega$.
			
			We start with the forward implication. Assume that $G\ihom H$, so that there exists $(\sigma,\tau)$ which embeds $\bar{f}^G$ into $\bar{f}^H$ by \cref{p:inaspecialcaseY}.
			We need to check that $y$ is isolated in $\omega^2+1$ iff $\sigma(y)$ isolated in $\omega^2+1$. This is not true for any embedding $\sigma$, but as $(\sigma,\tau)$ embeds $\bar{f}^G$ into $\bar{f}^H$ this follows from our first remark and the fact that $\tau(\omega)=\omega$ holds. Moreover, $\sigma$ being an embedding, it must fix the point $\infty$ (see \cref{fact:onCBrank}). It therefore follows from our hypotheses on $\rho$ that $\rho^{-1}(\{y\})$ embeds into $\rho^{-1}(\{\sigma(y)\})$ for all $y\in \omega^2+1$. Since $\rho$ is a pseudo-embedding, it follows from \cref{prop:PseudoEmbLiftingEmb} that there exists some embedding $\bar{\sigma}$ such that $\rho\circ \bar{\sigma}=\sigma\circ \rho$. Hence $(\bar{\sigma},\tau)$ embeds $\bar{f}^G\circ\rho$ into $\bar{f}^H\circ \rho$, as desired.
			
			For the backward implication, assume that $(\bar{\sigma},\tau)$ embeds $\bar{f}^G\circ\rho$ in $\bar{f}^H\circ \rho$ for some $G,H\in \CCG$. Let $L=\{x\in X \mid \text{$\rho(x)$ is limit in $\omega^2+1$}\}$. We start by noticing that $x\in L$ iff $\bar{\sigma}(x)\in L$ for all $x\in X$. By our first remark, we have $x\in L$ iff $\bar{f}^G\circ \rho(x)=\omega$, and similarly for $\bar{f}^H$. So as $(\bar{\sigma},\tau)$ embeds $\bar{f}^G\circ\rho$ into $\bar{f}^H\circ \rho$, we must have $x\in L$ iff $\bar{\sigma}(x)\in L$.
			
			Notice that as $\rho$ is a regular pseudo-embedding, $\rho$ actually restricts to embedding from $L$ into the limit points of $\omega^2+1$. Let $x_{(m,\omega)}\in X$ denote the unique point with $\rho(x_{(m,\omega)})=(m,\omega)$ and let $x_\infty\in X$ be the unique point such that $\rho(x_\infty)=\infty$ (if there exists one at all). The set $L$ is equal to either $\{x_{(m,\omega)}\mid m\in \N\}$ or to $\{x_\infty\}\cup \{x_{(m,\omega)}\mid m\in \N\}$. Since $\sigma$ restricts to an embedding of $L$ into itself, it must fix the point $x_\infty$ if it exists. In both cases $\bar{\sigma}$ induces an injective map $\sigma:\N\to\N$ such that $\bar{\sigma}(x_{(m,\omega)}))=x_{(\sigma(m),\omega)}$ for every $m$.

			Next we prove that this injective map $\sigma:\N\to\N$ is a homomorphism from $G$ to $H$. Assume that $n \mathrel{E^G} m$.
			We claim that whenever $(x_k)_k$ is an injective sequence of isolated points in $\omega^2+1$ converging to $(n,\omega)$, then by possibly going to a subsequence there exists a sequence $(x'_k)_k$ in $X$ with $\rho(x'_k)=x_k$ for all $k$ and such that $\rho(\bar{\sigma}(x'_k))_k$ is injective and converges to $(\sigma(n),\omega)$.
			To see this pick any sequence $(x'_k)_k$ in $X$ with $\rho(x'_k)=x_k$. Since $\rho$ is a pseudo-embedding and $(x'_k)$ is not $\rho$-trivial, we have $x'_k\to x_{(n,\omega)}$. Now since $\rho(x'_k)=x_k$ is isolated in $\omega^2+1$ so is $\rho(\bar{\sigma}(x'_k))$ and by continuity the sequence $\rho(\bar{\sigma}(x'_k))_k$ converges to the limit point $(\sigma(n),\omega)$. Therefore by possibly going to a subsequence, we can assume that $\rho(\bar{\sigma}(x'_k))_k$ is injective, which proves the claim.
			
			Since $n \mathrel{E^G} m$, by definition of
			$f^G$ there are injective sequences of isolated points $(x_k)_k$ and $(y_k)_k$ with $x_k\to (n,\omega)$, $y_k\to (m,\omega)$ and $\bar{f}^G(x_k)=\bar{f}^G(x_k)$ for all $k$. By possibly going simultaneously to subsequences, we use the claim to pick $(x'_k)_k$ and $(y'_k)_k$ in $X$ such that $\rho(x'_k)=x_k$, $\rho(y'_k)=y_k$, and setting $x''_k=\rho\circ\bar{\sigma}(x'_k)$ and $y''_k=\rho\circ\bar{\sigma}(y'_k)_k$, $(x''_k)_k$ and $(y''_k)_k$ are injective and converge to $(\sigma(n),\omega)$ and $(\sigma(m),\omega)$ respectively. Moreover for all $k$,
			\[
			f^H(x''_k)=\tau \circ \bar{f}^G\circ \rho(x'_k)=\tau \circ \bar{f}^G\circ \rho(y'_k)=f^H(y''_k).
			\]
			Therefore using \cref{lemIsoCtbleGraphs} we see that the two sequences $(x''_k)_k$ and $(y''_k)_k$ together witness that $\sigma(n) \mathrel{E^H} \sigma(m)$ holds.
		\end{proof}
		
		If $(C_n)_n$ is a sequence of mutually disjoint clopen sets in a space $X$ and $x\in X$, we write $C_n\to x$ if every neighborhood of $X$ contains $C_n$ for all but finitely many $n$. Notice that in this case, we have $x\notin C_n$ for all $n$, the set $F=\{x\} \cup \bigcup_n C_n$ is closed in $X$ and the map $\rho: F\to \omega+1$ given by $\rho(y)=n$ if $y\in C_n$ and $\rho(x)=\omega$ is a pseudo-embedding.
		
		\begin{lemma} \label{lem:UnctbleExistRegPE}
			Let $X$ be an uncountable Polish \zerodim{} space. Then there exists a regular pseudo-embedding $\rho:X\to\omega^2+1$.
		\end{lemma}
		\begin{proof}
			Fix some compatible metric for $X$ and recall that a point $x\in X$ is a \emph{condensation point} of $X$ if every neighborhood of $x$ is uncountable. Remember that $X$ is uncountable iff the set of condensation points of $X$ is uncountable (\cite[6.4, p.32]{kechris2012classical}). We fix some condensation point $x_\infty$ of $X$ and define a sequence $(C_m)_{m\in\omega}$ of  pairwise disjoint clopen sets such that
			\begin{enumerate}
				\item each $C_m$ is uncountable,
				\item $C_m\to x_\infty$, and
				\item $X=\{x_\infty\}\cup \bigcup_m C_m$.
			\end{enumerate}
			To this end, let $(V_n)_n$ be a decreasing sequence of clopen neighborhoods of $x_\infty$ with vanishing diameter and $V_0=X$. 
We claim that there exists a strictly increasing sequence $(n_i)_i$ of indices such that the clopen sets given by $C_0=X\setminus V_{n_0}$ and $C_{i+1}=V_{n_i} \setminus V_{n_{i+1}}$, for $i\geq 0$, are all uncountable. This sequence $(C_m)$ is clearly as desired. To see this, observe that there exists a condensation point $y\in V_0$ such that $y\neq x_\infty$ and so $y\notin V_{n_0}$ for some large enough $n_0$. Then $X\setminus V_{n_0}$ is uncountable as a clopen neighbourhood of the condensation point $y$. Next if $n_i$ is defined we can apply the same argument inside $V_{n_i}$ to get some $n_{i+1}>n_{i}$ such that $V_{n_i}\setminus V_{n_{i+1}}$ is uncountable, which completes the induction.

			We then pick a condensation point $x_{(m,\omega)}$ in $C_m$ for every $m$, and apply the same construction to $C_m$ and $x_{(m,\omega)}$ as we did above for $X$ and $x_\infty$. Hence we get for every $m$ a sequence $(C_{m,n})_n$ of uncountable clopen sets that partitions $C_m\setminus\{x_{(m,\omega)}\}$ and satisfies $C_{m,n}\to x_{(m,\omega)}$.
			
			We define $\rho:X\to \omega^2+1$ by $\rho(x_\infty)=\infty$, $\rho(x_{(m,\omega)})=(m,\omega)$ and $\rho(x)=(m,n)$ if $x\in C_{m,n}$. Observe that whenever $(x_n)_n$ in $X$ is not $\rho$-trivial, then $\rho(x_n)_n$ can only converge to some limit point $y$ of $\omega^2+1$ and by construction this happens if and only $x_n\to x$ for the unique point $x$ with $\rho(x)=y$. Hence $\rho$ is a pseudo-embedding. To see it is regular, notice that it clearly satisfies \cref{regularPE1} and \cref{regularPE2}. To see condition \cref{regularPE3}, remember that every \zerodim{} Polish space embeds into $2^\omega$ (see \cite[(7.8)]{kechris2012classical}) while  $2^\omega$ embeds into every uncountable Polish space \cite[(6.5)]{kechris2012classical}.
		\end{proof}
		
		\begin{lemma}\label{lem:CtbleExistRegPE}
			Let $X$ be a countable \zerodim{} space with infinitely many limit points. We can write $X$ as the union of two closed sets $F_0$ and $F_1$ with the following properties:
			\begin{enumerate}
				\item there is a regular pseudo-embedding $\rho:F_0\to\omega^2+1$, \label{lem:CtbleExistRegPE1}
				\item $F_0\cap F_1=\rho^{-1}(\{\infty\})$. \label{lem:CtbleExistRegPE2}
			\end{enumerate}
		\end{lemma}
		
		\begin{proof}
			We start by making the following remark. If $x$ is a limit point which is isolated in $\CB_1(X)$, then for every neighborhood $U$ of $x$ there exists a clopen set $C$ with $x\in C \subseteq U$, $C\cap \CB_1(X)=\{x\}$ and such that $C$ is homeomorphic to either $\omega+1$ or $[\omega]^2_+$. To see the last point, notice that if $C$ is clopen with $C\cap\CB_1(X)=\{x\}$ then $C$ is a Polish space with a single limit point and by \cref{prop:simpleCB2} it is homeomorphic to one of the three spaces $\omega\sqcup \omega+1$, $\omega+1$, $[\omega]^2_+$, and by shrinking $C$ we can actually make sure it is homeomorphic to either $\omega+1$ or $[\omega]^2_+$.
			
			Since $X$ is countable with infinitely limit points, $\CB_1(X)$ is infinite and countable. It follows that there are infinitely many isolated points in $\CB_1(X)$ and we can take an injective sequence $(x_n)_n$ of isolated points in $\CB_1(X)$. By going to a subsequence, we can assume that either $(x_n)_n$ converges to some $x_\infty$ in $X$, or that no subsequence of $(x_n)_n$ converges in $X$. 
			
			First suppose that no subsequence of $(x_n)_n$ converges in $X$. Consider mutually disjoint clopen neighborhoods $U_n$ of each point $x_n$. 
			By applying our first remark and possibly going to a subsequence, we get a sequence $(C_n)_n$ of mutually disjoint clopen sets with $C_n\cap \CB_1(X)=\{x_n\}$ and either $C_n\cong \omega+1$ for all $n$, or $C_n\cong [\omega]^2_+$ for all $n$. We set $F_0=\bigcup_n C_n$ and $F_1=X\setminus \bigcup_n C_n$. Since no subsequence of $(x_n)_n$ converges in $X$, it follows that $F_0$ is also closed. We define $\rho:F_0\to \omega^2+1$ as follows. In the first case, we identify $F_0$ with $\bigsqcup_n \omega+1$ and take $\rho$ to be the inclusion into $\omega^2+1$. In the second case, notice that $\pi:[\omega]^2_+\to \omega+1$ given by $\pi(\{k,l\})=k$ and $\pi(\emptyset)=\omega$ is a regular pseudo-embedding. So identifying $F_0$ with $\bigsqcup_n [\omega]^2_+$ we simply let $\rho=\bigsqcup_n \pi$. 
			
			Suppose now that $x_n\to x_\infty$. For some compatible metric on $X$ take mutually disjoint clopen neighborhoods $U_n$ of each $x_n$   whose diameters converges to $0$. This last condition ensures that $U_n \to x_\infty$. By applying our first remark and possibly going to a subsequence, we can find a sequence $(C_n)_n$ of clopen sets with $C_n\subseteq U_n$, $C_n\cap \CB_1(X)=\{x_n\}$ and either $C_n\cong \omega+1$ for all $n$, or $C_n\cong [\omega]^2_+$ for all $n$. We set $F_0=\{x_\infty\}\cup\bigcup_n C_n$ and $F_1=X\setminus \bigcup_n C_n$. Notice that $F_0$ is closed since each $C_n$ is clopen and $C_n \to x_\infty$. In the first case, we can take $\rho$ to be a homeomorphism from $F_0$ to $\omega^2+1$. In the second case, let $h_n:C_n \to [\omega]^2_+$ be a homeomorphism. We define $\rho:F_0\to\omega^2+1$ by $\rho(x)=(n,\pi\circ h_n(x))$ if $x\in C_n$ and $\rho(x_\infty)=\infty$.
			
			In all cases, the proof that $\rho$ is a pseudo-embedding is just as in \cref{lem:UnctbleExistRegPE}. Whenever $(x_n)_n$ in $F_0$ is not $\rho$-trivial, then $\rho(x_n)_n$ can only converge to some limit point $y$ of $\omega^2+1$ and, by construction, this happens if and only $x_n\to x$ for the unique point $x$ in $F_0$ with $\rho(x)=y$. Moreover, by construction $\rho$ is regular in all cases.
		\end{proof}

		We are now ready to prove the main result of this section.
		
		\begin{theorem}\label{t:thechaossideYY}
			If $X$ is a Polish \zerodim{} space with infinitely many limit points and $Y$ is not discrete, then $(\CCG,\ihom)$ reduces continuously to $(\C(X,Y),\emb)$.
		\end{theorem}
		\begin{proof}
			As $Y$ is not discrete, $\omega+1$ embeds into $Y$ and so by \cref{cor:embeddingrangesideY} it is enough to show the result for $Y=\omega+1$.
			Our continuous reduction $\G \to \C(X,\omega+1)$ is given by $G\mapsto \bar{f}^G\circ \pi$ for some well chosen $\pi:X\to\omega^2+1$. We start by noticing that as long as $\pi$ is continuous, then the map $\C(\omega^2+1,\omega+1)\to \C(X,\omega+1)$, $f\mapsto f\circ \pi$ is continuous. Suppose that $K\subseteq X$ is compact and $U\subseteq \omega+1$ is open. Then the set of $f\in \C(\omega^2+1,\omega+1)$ such that $f(\pi(K))\subseteq U$ is open in $\C(\omega^2+1,\omega+1)$, since $\pi(K)$ is compact. Since we proved that $G\mapsto \bar{f}^G$ is continuous in \cref{p:inaspecialcaseY}, the map $G \mapsto \bar{f}^G\circ \pi $ is a continuous from $\G$ to $\C(X,\omega+1)$.
			
			In case $X$ is uncountable we simply take $\pi$ to be the regular pseudo-embedding $\rho$ given by \cref{lem:UnctbleExistRegPE} and it follows from \cref{lem:RegularPE} that $G\ihom H$ iff $\bar{f}^G\circ \pi \emb \bar{f}^H\circ \pi$. So we indeed have a reduction in this case.
			
			In case $X$ is countable, \cref{lem:CtbleExistRegPE} allows us to write $X$ as the union of two closed sets $F_0$ and $F_1$ and grants us with the existence of a regular pseudo-embedding $\rho:F_0\to \omega^2+1$. Moreover either $F_0\cap F_1=\emptyset$ or $F_0\cap F_1=\{x_\infty\}$ for the unique point of $F_0$ such that $\rho(x_\infty)=\infty$. We take $\pi:X\to \omega^2+1$ to be the extension of $\rho$ which sends every point in $F_1$ to $\infty$.
			Notice that $\pi$ is continuous: if $F_0$ and $F_1$ are disjoint, $\pi$ is continuous as the topological sum of two continuous functions. In the other case, the proof is similar to the on we gave for \cref{LemWedgeEmb}.

			Let us check that this $\pi$ indeed yields a reduction from $\ihom$ to $\emb$. Suppose that $G\ihom H$. By \cref{lem:RegularPE}, there is some embedding $\bar{\sigma}:F_0\to F_0$ which embeds $\bar{f}^G\circ \rho$ into $\bar{f}^H\circ \rho$. Then we can extend $\bar{\sigma}$ to an embedding $\sigma$ of the whole space $X$ by letting $\sigma$ be the identity on $F_1$ (using \cref{LemWedgeEmb} in case $F_0$ and $F_1$ are not disjoint). Clearly $\sigma$ embeds $\bar{f}^G\circ \pi$ into $\bar{f}^H\circ \pi$, as desired.
			
			Conversely, assume that $\sigma:X\to X$ embeds $\bar{f}^G\circ \pi$ into $\bar{f}^H\circ \pi$. We claim that $\sigma$ restricts to an embedding $\bar{\sigma}:F_0 \to F_0$. This clearly implies that $\bar{\sigma}$ embeds $\bar{f}^G\circ \rho$ into $\bar{f}^H\circ \rho$, and so $G\ihom H$ by \cref{lem:RegularPE}. To see the claim, notice that for $I=\{x\in X \mid \text{$\rho(x)$ is isolated in $\omega^2+1$}\}$ we have $x\in I$ iff $\sigma(x)\in I$. Moreover, $F_0$ is equal to the closure of $I$. It follows that $x\in F_0$ iff $\sigma(x)\in F_0$, as desired.
		\end{proof}
		
		\begin{remark}\label{remContinuousInTheCodes}
				As we have seen, the space $\C(X,Y)$ is not Polish in general. There is however another way of looking at the reductions given in \cref{t:thechaossideYY} by considering continuous \textit{partial} functions in the codes. As $X$ and $Y$ are Polish \zerodim{}, they are homeomorphic to the spaces of infinite branch through pruned trees $S$ and $T$ respectively. Then every partial continuous function $\varphi^*$ from $X$ to $Y$ with Polish domain is coded by some monotone function $\varphi:S\to T$ \cite[(2.6)]{kechris2012classical}. The space of monotone functions from $S$ to $T$ is a closed subset of the Polish space $T^S$, hence Polish. Note that the continuous total functions form a coanalytic subset of this space. It is not hard to see that we can define in a continuous fashion for every $G\in \CCG$ a monotone map $\varphi^G:S \to T$ which codes the  continuous function $:X\to Y$ given for $G$ by our reduction. 
			\end{remark}

			\begin{corollary}
				Let $X$ and $Y$ be Polish \zerodim{} with $X$ compact. If $X$ has infinitely many and $Y$ is not discrete, then the restriction of the quasi-order of embeddability to $\C(X,Y)$ is analytic complete.
			\end{corollary}
			\begin{proof}
				We proved in \cref{EmbedIsAnalytic} that the quasi-order of embeddability is analytic on $\C(X,Y)$. By \cref{t:thechaossideYY}, we have a continuous reduction from $(\CCG,\ihom)$ to $(\C(X,Y),\emb)$. Hence this follows from the result of Louveau and Rosendal (\cref{LouveauRosendal}).
			\end{proof}

		\section{Order}\label{sec:ORDER}
		
		We now discuss the other half of the dichotomy and prove that the restriction of the embeddability quasi-order is in several cases a well-quasi-order. 
		To do so we appeal to a technical strengthening of the notion of well-quasi-order called better-quasi-order (\bqo{}) introduced by \textcite{nashwell}.

		\subsection{For a fistful of better-quasi-orders}
		
		For the reader's convenience, we give here some basic definitions and collect some results on \bqo{}s which are needed in the sequel. An introduction to the theory is given in \cite{simpsonbqo} and \cite{yann2017towardsbetter} while further results can be found for example in \cite{fbqo, vEMSbqo, LStRbqo, carroypequignot}.
		
		We let $[\omega]^\omega$ denote the set of infinite subsets of $\omega$ with the topology induced by the topology on $2^\omega$ under the identification of a set with its characteristic function. For $X\in[\omega]^\omega$, we let $[X]^\omega$ be the set of infinite subsets of $X$ endowed with the induced topology. If $Q$ is a set, a \emph{$Q$-multi-sequence} is a function $f$ with domain $[X_0]^\omega$ for some $X_0 \in [\omega]^\omega$ and range a countable subset of $Q$. A $Q$-multi-sequence $f$ is locally constant if $f^{-1}(\{q\})$ is open for every $q\in Q$. When $\leq_Q$ is a quasi-order on $Q$ we say that a $Q$-multi-sequence $f$ is \emph{bad} if for all $X\in \dom(f)$ we have $f(X)\nleq_Q f(X^+)$, where $X^+=X\backslash\{\min X\}$. A quasi-order on $Q$ is a \emph{better-quasi-order} (\bqo{}) if there are no bad locally constant $Q$-multi-sequences.
		
		Recall that a quasi-order $Q$ is \wqo{} if and only if there is no \emph{bad sequence} in $Q$, i.e. no sequence $(q_n)_n$ with $q_m\nleq q_n$ whenever $m<n$. Observe now that any bad sequence $(q_n)_n$ in $Q$ induces a locally constant bad $Q$-multi-sequence given by $f(X)=q_{\min X}$. Hence every \bqo{} is indeed a \wqo{}. Straightforward applications of the Galvin-Prikry Theorem \cite{JSL:9194679} show that every finite sum or finite product of \bqo{}s (and in particular any finite qo) is a \bqo{}.
		
		We say that a map $\sigma:Q\to P$ is a \emph{co-homomorphism} if for all $p,q$ in $Q$, $\sigma(p)\leq_P\sigma(q)$ implies $p\leq_Q q$. Notice that if there exists a co-homomorphism from $P$ to $Q$ and $Q$ is \bqo{}, then $P$ is a \bqo{} too.

		We will use a result due to \textcite{vEMSbqo} which establishes a very strong form of Fra\"issé's conjecture and refines the result of \textcite{laverfraisse}. To state this important result, we need to introduce the following strengthening of the \bqo{} property for classes of structures.

		\begin{definition}[\cite{LStRbqo}]\label{defn:preservebqos}
			Let $\mathcal{C}$ be a class of structures together with $\mathcal{C}$-morphisms between them. Assume that all identity maps are $\mathcal{C}$-morphisms and that $\mathcal{C}$-morphisms are closed under composition.\footnote{A class of structures as understood here corresponds exactly to a concrete category in the category theory sense.} Given a qo $Q$, the class $Q^{\mathcal{C}}$ of \emph{$Q$-labelled $\mathcal{C}$-structures} is given by
			\[
			Q^{\mathcal{C}}=\{f\mid f\mbox{ is a function, }\dom(f)\mbox{ is a }\mathcal{C}\mbox{-structure, } \im(f)\sub Q\},
			\]
			together with the following quasi-order:
			\begin{multline*}
			f_0\leq f_1 \Leftrightarrow\exists\,\mathcal{C}\mbox{-morphism } g:\dom(f_0)\rao\dom(f_1)\\
			\mbox{such that }\forall\,x\in\dom(f_0)\,f_0(x)\leq_Qf_1(g(x)).\end{multline*}
			We say that $\mathcal{C}$ \emph{preserves \bqo{}s} if for all \bqo{} $Q$ the class $Q^{\mathcal{C}}$ is also \bqo{}.
		\end{definition}

		Notice that if $\mathcal{C}$ is a class of structures, then the $\mathcal{C}$-morphisms induce a quasi-order where $C\leq C'$ for two structures $C$ and $C'$ in $\mathcal{C}$ iff there exists a $\mathcal{C}$-morphism from $C$ to $C'$. Of course, if a class $\mathcal{C}$ preserves \bqo{}s then in particular $\mathcal{C}$ is a \bqo{} under this quasi-order, as one easily sees by considering the one point \bqo{}.    
		
		The class $\lin$ consists of the countable linear orders. Given two linear orders $K$ and $L$, a $\lin$-morphism is an order preserving injection (or equivalently, an order embedding) from $K$ to $L$ that is continuous with respect to the order topology.
		
		\begin{theorem}[{\cite[Theorem 3.5]{vEMSbqo}}]\label{vEMS}
			The class $\lin$ preserves \bqo{}s.
		\end{theorem}
		
		Let us denote by $\mathcal{Q}$ the class consisting of the single linear order $\Q$ together with topological embeddings as $\mathcal{Q}$-morphisms. Notice that a map between two linear orders is order preserving and injective iff it is an embedding of linear orders. Moreover if $j:K\to L$ is an order embedding between linear orders equipped with the order topology, then the direct image of any open set in $K$ is relatively open in the range of $j$. Therefore, any $\lin$-morphism $f:\Q\to \Q$ is in particular a topological embedding. We state the following direct corollary of \cref{vEMS} for later use.
		
		\begin{corollary}\label{corvEMS}
			The class $\mathcal{Q}$ preserves \bqo{}s.
		\end{corollary}

		We denote by $\POD$ the class of Polish \zerodim{} spaces quasi-ordered by topological embeddability. We recall \cite[Theorem 1.3]{carroy2}:
		
		\begin{theorem}\label{carroyPOD}
			The quasi-order $\POD$ is a \bqo{}.
		\end{theorem}
		
		In the sequel, if $Q$ is a quasi-order, then we denote by $Q_\star$ the quasi-order obtained from $Q$ by adding a new element $\star$ which is incomparable with every other element of $Q$. Of course, $Q$ is \bqo{} iff $Q_\star$ is \bqo{}.

		\subsection{For a few well-quasi-orders more}
		
		Since any countable metrizable space is homeomorphic to a closed subspace of $\Q$ (see \cite[(7.12)]{kechris2012classical}), we have
		\begin{fact}\label{fact:fctctblerange}
			For every continuous function $f$ with a \zerodim{} Polish domain $X$ and countable image there is $f^\Q\in\C(X,\Q)$ with $f\emb f^{\Q}$ and $f^{\Q}\emb f$.
		\end{fact}
		
		We first show that when $X$ is discrete then embeddability is a \bqo{} on $\C(X,Y)$. Since every continuous function from a discrete space is locally constant, this follows from the following more general result.
		
		\begin{theorem}\label{t:locallyconstant}
			The class $\LC$ of locally constant functions with Polish \zerodim{} domain is better-quasi-ordered by topological embeddability.
		\end{theorem}
		
		\begin{proof}
			Let $\LC_\Q$ denote the class of locally constant functions $f:X\to \Q$ where $X$ is Polish \zerodim{} quasi-ordered by embeddability. By \cref{fact:fctctblerange}, it suffices to show that $\LC_\Q$ is \bqo{}.
			
			We define for every $f\in \LC_\Q$ a function $\Lambda_f:\Q\to \POD_\star$ with the following property: whenever $\tau:\Q\to \Q$ witnesses $\Lambda_f\leq \Lambda_g$ in $\POD_\star^{\mathcal{Q}}$ we can find an embedding $\sigma:X\to X$ such that $(\sigma,\tau)$ is a witness for $f\emb g$. Combining \cref{corvEMS} and \cref{carroyPOD}, we get that $\POD_\star^{\mathcal{Q}}$ is a \bqo{}. This makes $f\mapsto \Lambda_f$ a co-homomorphism, which implies that $(\LC_\Q,\emb)$ is \bqo{}.
			
			Given $f:X\to\Q$, let
			\begin{align*}
			\Lambda_f:\Q&\longrightarrow\POD_\star \\
			q&\longmapsto\begin{cases}
			f^{-1}(\{q\})&\mbox{if }q\in\im f\\
			\star&\mbox{otherwise.}
			\end{cases}
			\end{align*}
			Notice that for all $q\in \im f$ the set $f^{-1}(\{q\})$ is a clopen subset of $X$, and as such it is a Polish \zerodim{} space.
			
			Take $f:X\to \Q$ and $g:Y\to\Q$ locally constant with $X$ and $Y$ Polish \zerodim{} and suppose that $\Lambda_f\leq \Lambda_g$ holds in $\POD_\star^\mathcal{Q}$ as witnessed by some embedding $\tau:\Q\to \Q$.
			
			By the role of $\star$, we have $q\in\im f$ if and only if $\tau(q)\in\im g$ for all $q\in Q$ and so $\tau$ restricts to an embedding $:\im f \to \im g$. Moreover for every $q\in \im f$ we can pick an embedding $\sigma_q:f^{-1}(\{q\})\to g^{-1}(\{\tau(q)\})$. As $f$ and $g$ are locally constant, $f^{-1}(\{q\})$ is clopen in $X$ and $g^{-1}(\{\tau(q)\})$ is clopen in $Y$. This easily implies that $\bigcup_{q\in \im f} \sigma_q$ is an embedding of $X$ into $Y$. As $\tau \circ f=g\circ \sigma$ clearly holds, we get $f\emb g$ via $(\tau,\sigma)$ as desired.
		\end{proof}

		\begin{theorem}\label{t:finitelymanylimitpoints}
			If $X$ is a locally compact Polish \zerodim{} space with finitely many non-isolated points then $(\C(X,Y), \emb)$ is a \bqo{}.
		\end{theorem}
		\begin{proof}
			We start by reducing the problem in several steps.
			Any Polish space with only finitely many limit points is in particular countable, so any function with such domain has a countable image, so once again it is enough to treat the case of rational valued functions.  We prove that topological embeddability on $\C(X,\Q)$ is a \bqo{}.
			
			If $D$ is a discrete space, then every continuous function from $D$ is locally constant.
			Therefore as in the proof\footnote{As $D$ is discrete, any injective function $\sigma:D\to D$ is an embedding and so we do not really need $\POD$ in this case. We could use $\omega+1$ instead and let $\Lambda_f(q)=|f^{-1}(\{q\}|$ if $q\in \im f$ and $\star$ otherwise.} of \cref{t:locallyconstant} we have map $\C(D,\Q)\to \POD_\star^\mathcal{Q}$, $f \mapsto \Lambda_f$ such that if $\tau$ witnesses $\Lambda_f\leq \Lambda_g$ in $\POD_\star^\mathcal{Q}$ then $f\emb g$ as witnessed by $(\sigma,\tau)$ for some embedding $\sigma:D\to D$.

			We will define a \bqo{} $\QOcolor$ and associate to every $f\in \C(\omega+1,\Q)$ a function $\Gamma_f:\Q\to \QOcolor$ such that if $\tau$ witnesses $\Gamma_f\leq \Gamma_g$ in $\QOcolor^\mathcal{Q}$, then $f\emb g$ via $(\sigma,\tau)$ for some $\sigma$.
			
			Assuming this for now, we show that this enough to conclude. Suppose first that $X$ has a unique limit point. Applying \cref{prop:simpleCB2} and using the fact that $X$ is locally compact, we have $X=D\sqcup \omega+1$ for some (possibly empty) discrete space $D$. Consider the map $\Gamma^X:\C(X,\Q)\to (\POD_\star\times \QOcolor)^\mathcal{Q}$ given by $\Gamma^X(f)(q)=(\Lambda_{f_0}(q),\Gamma_{f_1}(q))$ for every $q\in \Q$ where $f_0=f\rest{D}$ and $f_1=f\rest{\omega+1}$. Then, as finite products of \bqo{}s are \bqo{}, $(\POD_\star\times \QOcolor)^\mathcal{Q}$ is \bqo{} by \cref{corvEMS}. Moreover  $\Gamma^X$ is also a co-homomorphism. To see this, notice that if $\tau$ witnesses $\Gamma^X(f)\leq \Gamma^X(g)$ in $(\POD_\star\times \QOcolor)^\mathcal{Q}$ then there exist two embeddings $\sigma_0:D\to D$ and $\sigma_1:\omega+1\to\omega+1$ such that $(\sigma_i,\tau)$ embeds $f_i$ in $g_i$ for $i=0,1$. Hence, clearly $(\sigma_0\cup \sigma_1,\tau)$ embeds $f$ in $g$.
			
			Now suppose that $X$ is locally compact with $n+1$ limit points. Then by \cref{fact:onCBrank}, $X=\bigsqcup_{i\leq n} X_i$ where each $X_i$ has a single limit point and is locally compact too. Consider the map $\Gamma^X:\C(X,\Q)\to (\prod_{i\leq n}(\POD_\star\times \QOcolor))^\mathcal{Q}$ given by $\Gamma^X(f)(q)=(\Gamma^{X_i}(f_i)(q))_{i\leq n}$ for every $q\in \Q$ where $f_i=f\rest{X_i}$. As before $(\prod_{i\leq n}(\POD_\star\times \QOcolor))^\mathcal{Q}$ is \bqo{} and the proof that $\Gamma^X$ is a co-homomorphism is similar.
			
			We take $\QOcolor$ to be the disjoint union $(\omega+1)\sqcup \omega \sqcup \{\star\}$ quasi-ordered by $(i,p)\leq (j,q)$ iff $i=j$ and $p\leq q$. Then $\QOcolor$ is \bqo{} as a finite sum of \bqo{}s.
			Let $f:\omega+1\to \Q$ be continuous. Notice that if $ q\neq f(\omega)$, then by continuity of $f$ the set $f^{-1}(\{q\})$ must be finite. Define
			\begin{align*}
			\Gamma_f:\Q&\longrightarrow \QOcolor  \\
			q&\longmapsto\begin{cases}
			(0,|f^{-1}(\{q\})|)&\mbox{if $f(\omega)=q$,}\\
			(1,|f^{-1}(\{q\})|)&\mbox{if $f(\omega)\neq q$, but $q\in \im f$, and}\\
			(2,\star)& \mbox{otherwise.}
			\end{cases}
			\end{align*}
			Suppose that $\tau:\Q\to \Q$ is an embedding witnessing $\Gamma_f\leq \Gamma_g$. By the role of $\star$ $\tau$ necessarily restricts to an embedding $:\im f \to \im g$. Moreover we must have $\tau(f(\omega))=g(\omega)$ and the cardinality of $f^{-1}(\{f(\omega)\})$ is lower or equal to that of $g^{-1}(\{g(\omega)\})$. Therefore taking any injection $\sigma_\omega: f^{-1}(\{f(\omega)\}) \to g^{-1}(\{g(\omega)\})$ with $\sigma_\omega(\omega)=\omega$ we have an embedding. Finally for every $q\in \im f$ with $q\neq f(\omega)$ there is an injection $\sigma_q: f^{-1}(\{q\})\to g^{-1}(\{\tau(q)\})$. By taking the union of all the $\sigma_q$ together with $\sigma_\omega$, we obtain an injection $\sigma:\omega+1\to\omega+1$ such that $\sigma(\omega)=\omega$, therefore $\sigma$ is an embedding. Since clearly $g\circ \sigma=\tau\circ f$, it follows that $(\sigma,\tau)$ witnesses $f\emb g$.
		\end{proof}


		\section{Maximal elements in Baire classes} \label{sec:RankEmbed}
		
		In this section, we investigate the existence of maximal elements for embeddability in the Baire classes. We start by looking at the very specific case  of the class of continuous functions.
		
		\subsection{Continuous functions}
		
		Consider the Hilbert cube $[0,1]^\omega$ and let $\proj:[0,1]^\omega\times[0,1]^\omega \to [0,1]^\omega$ be the projection $\proj(\beta,\gamma)=\gamma$.
		
		\begin{proposition}\label{prop:ProjMaximum}
			The map $\proj$ is a maximum for the embeddability quasi-order on continuous functions between separable metrizable spaces. In particular, it is a maximum on continuous functions between Polish spaces.
		\end{proposition}
		
		\begin{proof}
			Let $f:X\to Y$ be a continuous function between two separable metrizable spaces. The Hilbert cube is universal for separable metrizable spaces (\cite[(4.14)]{kechris2012classical}), so there exists embeddings $\sigma':X\to [0,1]^\omega$ and $\tau:Y\to [0,1]^\omega$. Define $\sigma:X \to [0,1]^\omega\times[0,1]^\omega$ by $\sigma(x)=(\sigma'(x),\tau(f(x))$. As $f$ is continuous and $X,Y$ are separable metrizable, one easily sees that $\sigma$ is an embedding. Moreover we have $\proj(\sigma(x))=\tau(f(x))$ for all $x\in X$. Hence $(\sigma,\tau)$ embeds $f$ into $\proj$, as desired.
		\end{proof}


		Remark that using the universal property of the Baire space (see \cite[(7.8)]{kechris2012classical}), the same argument shows that the corresponding projection $\proj':\omega^\omega\times \omega^\omega\to \omega^\omega$ is a maximum for embeddability on continuous functions between metrizable \zerodim{} spaces.
		
		We briefly mention that while the class of continuous functions admits a maximum, the class of discontinuous ones admits a two-elements basis. Let us call $d_0:\omega+1\to 2$ the characteristic function of $\{\omega\}$ and let $d_1:\omega+1\to \omega$ denote the bijection given by $d_1(\omega)=0$ and $d_1(n)=n+1$.
		
		\begin{proposition}\label{prop:contvsdiscont}
			Let $f:X\to Y$ be any function between separable metrizable spaces. Then exactly one of the two following holds:
			\begin{enumerate}
				\item either $f\emb \proj$, or
				\item either $d_0\emb f$ or $d_1\emb f$.
			\end{enumerate}
			Moreover $\proj$, $d_0$ and $d_1$ form an antichain.
		\end{proposition}
		
		\begin{proof}
			Let us first note that $\proj,d_0$ and $d_1$ are pairwise incomparable. As their images have distinct cardinalities, we have $\proj\not\emb d_0$, $\proj\not\emb d_1$ and $d_1\not\emb d_0$. Since $\proj$ is continuous while $d_0$ and $d_1$ are not, it follows that $d_0\not\emb \proj$ and $d_1\not\emb \proj$. Finally as $d_1$ is injective and $d_0$ is not, we get $d_0\not\emb d_1$.
			
			Since $\proj$ is maximal among continuous functions between separable metrizable spaces by \cref{prop:ProjMaximum}, it remains to see that any discontinuous function $f:X\to Y$ between separable metrizable spaces embeds either $d_0$ or $d_1$. So suppose that $f:X\to Y$ is not continuous and take a sequence $(x_n)_n$ in $X$ such that $x_n\to x$ but $f(x_n)\not\to f(x)$ in $Y$. By passing to a subsequence we can assume that for all $n\in\N$ we have $x_n\neq x$, and since $x_n\to x$, by going to a further subsequence we can also assume that $(x_n)_n$ is injective. We distinguish two cases. First, suppose that $f$ takes only finitely many values on $(x_n)_n$. Then we can go to a subsequence on which $f$ is constant and we set $\tau_0:2\to Y$, $\tau_0(0)= f(x_0)$ and $\tau_0(1)=f(x)$. Otherwise $f$ takes infinitely many values on $(x_n)_n$ and we can find a subsequence on which $f$ is injective. In this case, we define $\tau_1:\omega \to Y$ by $\tau_1(0)=f(x)$ and $\tau_1(n+1)=f(n)$.
			In both cases, we let $\sigma:\omega+1\to Y$ be given by $\sigma(n)=x_n$ and $\sigma(\omega)=x$.
			In the first case $(\sigma,\tau_0)$ embeds $d_0$ in $f$ and in the second case $(\sigma,\tau_1)$ embeds $d_1$ in $f$.
		\end{proof}

		\subsection{Ranks and embeddability}
		
		For separable metrizable spaces $X$ and $Y$ and a countable ordinal $\xi\geq 1$, recall that a function $f:X\to Y$ is Baire class $\xi$ if the preimage by $f$ of any open subset of $Y$ is $\mathbf{\Sigma}^0_{\xi+1}$ in $X$ (\cite[24.3]{kechris2012classical}). We let $\mathcal{B}_\xi(X,Y)$ denote the class of Baire class $\xi$ functions from $X$ to $Y$, and simply write $\mathcal{B}_\xi(X)$ when $X=Y$. In contrast with the class of continuous functions, we prove in this subsection that there are no greatest elements with respect to $\emb$ in the classes of functions $\mathcal{B}_\xi(X,Y)$ if $1 \leq \xi < \omega_1$, if $X$ is an uncountable Polish space and $Y$ is an arbitrary Polish space with $|Y| \geq 2$. Moreover if $X$ is a \zerodim{} Polish space or a euclidean space, then we further show that there are no maximal elements in $\mathcal{B}_\xi(X)$ for $2 \leq \xi <\omega_1$. This will be proved using a modified version of the ranks defined on the class $\mathcal{B}_1$ by Bourgain \cite{bourgain1978convergent}, Zalcwasser \cite{zalc1930}, Gillespie and Hurwicz \cite{gillespie1930sequences} and others and extensively studied by Kechris and Louveau \cite{kechris1990classification}. The theory of these ranks has been generalized to higher Baire classes in \cite{ekv}, \cite{viktor} and \cite{marci}. Usually these ranks are defined solely for real valued functions, but in our investigation we define a rank in a more general context, resembling the so called separation rank of Bourgain.
		
		We fix for the rest of this section two Polish spaces $X$ and $Y$ with $X$ uncountable and $|Y| \geq 2$.
		
		\begin{definition}
			If $A$ is a $\mathbf{\Delta}^0_{\xi+1}$ subset of $X$ then it can be expressed as a transfinite difference of a countable decreasing sequence of $\mathbf{\Pi}^0_{\xi}$ sets, i.e. \[A=\bigcup_{\eta<\lambda, \eta \text{ even}} F_{\eta} \setminus F_{\eta+1},\] where $F_\eta \in \mathbf{\Pi}^0_{\xi}$ (see \cite[22.E]{kechris2012classical}). We call the minimal length of such a sequence the \textit{rank} of $A$, and denote it by $\rho_{\xi,0}(A)$.
			
			If $A_1$ and $A_2$ are disjoint $\mathbf{\Pi}^0_{\xi+1}$ sets they can be separated by a $\mathbf{\Delta}^0_{\xi+1}$ set, i.e. there is a $\mathbf{\Delta}^0_{\xi+1}$ set $A$ so that $A_1 \subseteq A \subseteq X \setminus A_2$. We denote by $\rho_{\xi,1}(A_1,A_2)$ the minimal rank of such a separating set.
		\end{definition}

		From this definition and the Hausdorff-Kuratowski analysis (see \cite[Theorem 22.27]{kechris2012classical}) we get the following.
		
		\begin{fact}
			\label{f:unbounded}
			For each disjoint $A_1, A_2 \in \mathbf{\Pi}^0_{\xi+1}$ the rank $\rho_{\xi,1}(A_1,A_2)<\omega_1$. It is well known that the rank $\rho_{\xi,0}$ is unbounded in $\omega_1$ on the class  $\mathbf{\Delta}^0_{\xi+1}(X)$ (see e.g. \cite[Theorem 4.3]{ekv}). Hence, for every ordinal $\alpha<\omega_1$, taking sets $A$ and $X \setminus A$, so that $A \in \mathbf{\Delta}^0_{\xi+1}$ and $\rho_{\xi,0}(A)>\alpha$ we have $\rho_{\xi,1}(A,X \setminus A)>\alpha$.
		\end{fact}
		
		\begin{definition}
			Let $f \in \mathcal{B}_\xi(X,Y)$. Define
			\[\rho_{\xi}(f)=\sup_{y_1,y_2 \in Y \text{ distinct}} \rho_{\xi,1}(f^{-1}(y_1),f^{-1}(y_2)).\]
		\end{definition}
		
		Notice that this definition makes sense: if $f \in \mathcal{B}_\xi(X,Y)$ then the preimages $f^{-1}(y_1), f^{-1}(y_2)$ are disjoint $\mathbf{\Pi}^0_{\xi+1}(X)$ sets. Moreover, the following easy claims show that this rank is meaningful on the class $\mathcal{B}_\xi(X,Y)$.
		
		\begin{lemma}\label{c:boundedInOmega1}
			If $f \in \mathcal{B}_\xi(X,Y)$ then $\rho_\xi(f)<\omega_1$.
		\end{lemma}
		\begin{proof}
			Let $\mathcal{B}=\{U_n: n \in \omega\}$ be a countable basis in $Y$. Define
			\[\rho'_{\xi}(f)=\sup \{\rho_{\xi,1}(f^{-1}(\overline{U}),f^{-1}(\overline{V})):U,V \in \mathcal{B}, \overline{U} \cap \overline{V}= \emptyset\}.\]
			Since $f \in \mathcal{B}_\xi(X,Y)$, if $\overline{U} \cap \overline{V}=\emptyset$ then the sets $f^{-1}(\overline{U}),f^{-1}(\overline{V})$ are disjoint $\mathbf{\Pi}^0_{\xi+1}(X)$ sets so the supremum is taken over a countable set of countable ordinals. Therefore, $\rho'_{\xi}(f)<\omega_1$.
			
			We show that $\rho_\xi(f) \leq \rho'_\xi(f)$, which is enough to prove the claim. Indeed, for every $y_1,y_2 \in Y$ distinct there exist $U,V \in \mathcal{B}$ with $\overline{U} \cap \overline{V}=\emptyset$ and $y_1 \in U$, $y_2 \in V$. But then if for a $\mathbf{\Delta}^0_{\xi+1}(X)$ set $A$ we have     
			$f^{-1}(\overline{U}) \subseteq A \subseteq X \setminus f^{-1}(\overline{V})$
			then of course
			$f^{-1}(y_1) \subseteq A \subseteq X \setminus f^{-1}(y_2).$
			Thus, by definition $\rho_{\xi,1}(f^{-1}(y_1),f^{-1}(y_2)) \leq \rho_{\xi,1}(f^{-1}(\overline{U}),f^{-1}(\overline{V}))$ and in turn $\rho_{\xi}(f) \leq \rho'_{\xi}(f)$.
		\end{proof}
		
		\begin{lemma}
			\label{c:unbounded}
			We have
			\[\sup\{ \rho_\xi(f)\mid f \in \mathcal{B}_\xi(X,Y)\}=\omega_1,\] in fact the rank of functions with range of cardinality at most $2$ is unbounded in $\omega_1$.
		\end{lemma}
		\begin{proof}
			By \cref{f:unbounded} for every $\alpha<\omega_1$ there exists a $\mathbf{\Delta}^0_{\xi+1}(X)$ set $A$ so that $\rho_{\xi,1}(A,X \setminus A)>\alpha$. Choosing $y_1,y_2 \in Y$ distinct and letting
			$
			f(x)=
			\begin{cases}
			y_1, &\text{ if } x \in A\\
			y_2, &\text{ otherwise}
			\end{cases}
			$
			we have that $\rho_\xi(f)=\rho_{\xi,1}(A,X\setminus A)>\alpha$.
		\end{proof}
		
		\begin{proposition}
			\label{p:ranks}
			If $g \in \mathcal{B}_{\xi}(X,Y)$, $f \in \mathcal{B}_{\xi}(X',Y')$ and $g \emb f$ then $\rho_\xi(g) \leq \rho_\xi(f)$.
		\end{proposition}
		
		\begin{proof}
			Let $\sigma$ and $\tau$ witness the embedding $g \emb f$. Let $y_1,y_2 \in Y$ be distinct. It is enough to prove that $\rho_{\xi,1}(g^{-1}(y_1),g^{-1}(y_2)) \leq \rho_{\xi}(f)$. Since $\tau$ is an embedding, $f^{-1}(\tau(y_1))$ and $f^{-1}(\tau(y_2))$ are disjoint and there exists a set $A \in \mathbf{\Delta}^0_{\xi+1}(X')$ separating them such that $\rho_{\xi,0}(A) \leq \rho_\xi(f)$. Therefore, there exists a sequence $(F_\eta)_{\eta<\rho_\xi(f)}$ of $\mathbf{\Pi}^0_\xi(X')$ sets with
			\[A=\bigcup_{\eta<\rho_\xi(f), \eta \text{ even}} F_{\eta} \setminus F_{\eta+1}.\]
			Then clearly \[\sigma^{-1}(A)=\bigcup_{\eta<\rho_\xi(f), \eta \text{ even}} \sigma^{-1}(F_{\eta}) \setminus \sigma^{-1}(F_{\eta+1}),\]
			and since $\sigma$ is a topological embedding, the sequence $(\sigma^{-1}(F_{\eta}))_{\eta<\rho_\xi(f)}$ is a decreasing sequence of $\mathbf{\Pi}^0_\xi(X)$ sets, so $\rho_{\xi,0}(\sigma^{-1}(A)) \leq \rho_\xi(f)$. From $\tau \circ g=f \circ \sigma$ we get \[x \in g^{-1}(y_1) \iff \tau(g(x))=\tau(y_1) \iff\]\[ f(\sigma(x))= \tau(y_1) \iff  x \in \sigma^{-1}(f^{-1}(\tau(y_1))),\]
			so $g^{-1}(y_1)=\sigma^{-1}(f^{-1}(\tau(y_1)))$. But, since $A$ separates $f^{-1}(\tau(y_1))$ and $f^{-1}(\tau(y_2))$ we have that $\sigma^{-1}(A)$ separates $\sigma^{-1}(f^{-1}(\tau(y_1)))=g^{-1}(y_1)$ and $\sigma^{-1}(f^{-1}(\tau(y_2)))=g^{-1}(y_2)$. Thus,
			$\rho_{\xi,1}(g^{-1}(y_1),g^{-1}(y_2)) \leq \rho_{\xi,0}(A) \leq \rho_{\xi,0}(\sigma^{-1}(A)) \leq  \rho_{\xi}(f)$, which finishes the proof.
		\end{proof}
		
		As an immediate corollary we get the following weaker version of the main theorem of this section.
		\begin{corollary}
			\label{c:rank0}
			There is no $\emb$-greatest element among Baire class $\xi$ functions from $X\to Y$, i.e. $f\in\mathcal{B}_\xi(X,Y)$ and $g\emb f$ for every $g\in \mathcal{B}_\xi(X,Y)$.
		\end{corollary}
		\begin{proof}
			By contradiction suppose that $f$ is a $\emb$-greatest element in $\mathcal{B}_\xi(X,Y)$. By \cref{c:boundedInOmega1} the rank of $f$ is some countable ordinal $\alpha$. By \cref{c:unbounded}, there exists a Baire class $\xi$ function $g:X\to Y$ whose rank is some countable ordinal $\beta$ strictly bigger than $\alpha$. Now as $f$ is maximal we have $g\emb f$, so by \cref{p:ranks} we have $\beta\leq \alpha$, a contradiction.
		\end{proof}
		
		Now we are ready to prove the main theorem of this section; it implies in particular \cref{thmINTRO3}.
		\begin{theorem}
			\label{t:rank}
			Let $1 \leq \xi <\omega_1$, $X, Y$ be Polish spaces such that there exists a $h:X \to X$ embedding with $|X \setminus h(X)|= \mathfrak{c}$ and $h(X)$ being $\mathbf{\Delta}^0_2$ and $|Y| \geq 2$. There is no $\emb$-maximal function in $\mathcal{B}_\xi$. In particular, this is true if $X$ is an uncountable \zerodim{} Polish space or a euclidean space.
		\end{theorem}
		\begin{proof}
			Let $f$ be an arbitrary function in $\mathcal{B}_\xi(X,Y)$. Take a perfect set $P \subseteq X  \setminus h(X)$. As in \cref{c:unbounded} take a $\Delta^0_{\xi+1}(P)$ set $A$ with $\rho_{\xi,0}(A)>\rho_\xi(f)$ and define a function $g_0:P \to Y$ by picking two distinct points $y_1$ and $y_2$ in $Y$ and letting $g_0(x)=y_1$ if $x\in A$ and $g_0(x)=y_2$ otherwise. Let $y \in Y$ be arbitrary and define
			
			$
			g(x)=
			\begin{cases}
			g_0(x), &\text{ if $x \in P$}\\
			f \circ h^{-1}(x), &\text{ if $x \in \im h$}\\
			y, &\text{ otherwise.}
			\end{cases}
			$
			
			It is easy to see that $g$ is a Baire-$\xi$ function: for an open $U \subseteq Y$ we have that $g^{-1}(U)$ is the union of at most three sets from the collection $h(f^{-1}(U)),P, X \setminus (P \cup h(A)), A, P \setminus A$. Clearly, using that $\im h$ is $\mathbf{\Delta}^0_2$ and $h$ is a homeomorphism, we get that all these sets are $\mathbf{\Sigma}^0_{\xi+1}$.
			
			On the one hand, $g \circ h=f \circ h^{-1} \circ h=f$, so $f \emb g$. On the other hand if a set $B$ separates $g^{-1}(y_1)$ from $g^{-1}(y_2)$ then $B \cap P=A$. By the fact that $P$ is closed, intersecting each element of a transfinite sequence of $\mathbf{\Pi}^0_{\xi+1}$ sets in $X$ with $P$ we still get such a sequence with at most the same length. Hence if we can express $B$ as a union of differences of a sequence of length $\alpha$, then this gives an expression of $A$ in $P$ with a union of the differences of a sequence of length at most $\alpha$. Thus, by the definition of our rank, $\rho_\xi(g) \geq \rho_{\xi,0}(A)>\rho_\xi(f)$. Then by \cref{p:ranks} we get $g \not \emb f$, so $f$ is not maximal.
		\end{proof}
		
		\begin{remark}
			It is not hard to see that \cref{t:rank} cannot be extended to arbitrary Polish spaces: let $X$ be a so called Cook continuum \cite{cook1967continua}, then every continuous non-constant function from $X$ to $X$ is the identity (see \cite{gelbaum2003counterexamples}). Hence, if we consider the functions from $X$ to $X$ then the embedding relation trivializes, that is, $g \emb f$ if and only if $f=g$. Thus, every function is maximal.
		\end{remark}
		
		
		\section{Concluding remarks and open questions}\label{section:last}

		Combining \cref{t:thechaossideYY,t:locallyconstant,t:finitelymanylimitpoints} we have obtained a dichotomy for the quasi-order of embeddability on $\C(X,Y)$ when $X$ is locally compact:
		\begin{enumerate}
			\item either it continuously reduces an analytic complete quasi-order, or
			\item it is a better-quasi-order.
		\end{enumerate}
		
		Firstly, when $X$ is locally compact and not compact, we do not know in general the exact complexity of the quasi-order of embeddability on the Polish space $\C(X,Y)$. While a direct calculation gives a $\mathbf{\Sigma}^1_2$ definition, there are examples where embeddability is not analytic. We now briefly explain this fact. Notice that a function $f:X\to X$ is an embedding iff $f\emb \text{id}_X$. Hence if $\emb$ is analytic on $\C(X,X)$, then so is the set of embeddings of $X$ into itself.
		
		\begin{proposition}
			The set of embeddings of $\N\sqcup 2^\omega$ into itself is co-analytic complete for the compact-open topology.
		\end{proposition}
		\begin{proof}
			We give a continuous reduction from the wellfounded trees on $\omega$ to the set of embeddings. For any tree $T$ on $\omega$, we define a continuous function $f_T: \omega^{<\omega} \sqcup 2^\omega \to \omega^{<\omega} \sqcup 3^\omega$ as follows. If $s=(s_0,\ldots,s_n)\in \omega^{<\omega}$, then $f_T(s)=s$ if $s\notin T$ and $f_T(s)=0^{s_0}10^{s_1}1 \cdots 0^{s_n}1 2^\omega$ if $s\in T$. If $\alpha\in 2^\omega$, then $f_T(\alpha)=\alpha$. Clearly $f_T$ is continuous and injective for all tree $T$. Suppose that $T$ has an infinite branch $\gamma$, then $(\gamma\rest{n})_n$ does not converge in the domain, while $f_T(\gamma\rest{n})_n$ converges to some $f_T(\alpha)$ in the range. Hence $f_T$ is not an embedding. When $T$ is wellfounded, then $f_T$ is easily seen to be an embedding.
			
			To see that the map $T\mapsto f_T$ is continuous, notice that any compact subset $K$ of $\omega^{<\omega} \sqcup 2^\omega$ contains only finitely many points in $\omega^{<\omega}$. Therefore the condition $f_T(K)\subseteq U$ for some open set $U$ only depends on a finite piece of $T$.
		\end{proof}
		
		It follows that embeddability is not analytic on $\C(\N\sqcup 2^\omega,\N\sqcup 2^\omega)$. Moreover if $Y$ is uncountable and not compact, then $\N\sqcup 2^{\omega}$ embeds in $Y$. Therefore by \cref{cor:embeddingrangesideY}, it follows that embeddability is not analytic on $\C(\N\sqcup 2^\omega, Y)$ either. By \cref{t:thechaossideYY}, embeddability is a $\mathbf{\Sigma}^1_1$-hard quasi-order, but we do not know its exact complexity.

		Secondly, we believe that one can drop the local compactness assumption on $X$ in the above dichotomy. From the proof of \cref{t:finitelymanylimitpoints} it appears that an essential step in this direction is to prove the following conjecture.
		
		\begin{conjecture}
			The quasi-order $(\C([\omega]^2_+,\Q),\emb)$ is \bqo{}.
		\end{conjecture}

		Let us now make a remark on embeddability beyond the case of continuous functions. A direct application of \cite[Theorem 3.3]{vEMSbqo} shows that when $X$ is a compact \zerodim{} Polish space and $Y$ is finite the Baire class $1$ functions from $X$ to $Y$ are \bqo{} under embeddability. This begs to consider the following much broader problem.
		
		\begin{question}
			If $X$ has only finitely many limit points, or if $Y$ is discrete, are the Baire class 1 functions from $X$ to $Y$ better-quasi-ordered by embeddability?
			What about all Borel functions from $X$ to $Y$?
		\end{question}

		In a different direction, recall that a \emph{basis} for some upward-closed subset $S$ of a quasi-order $Q$ is a subset $B\subseteq S$ such that $q\in S$ iff $\exists b\in B \ b\leq q$. As a matter of fact, a quasi-order is a well-quasi-order if and only if every upward-closed subset admits a finite basis. As we proved that embeddability on $\C(X,Y)$ is in most cases very complicated and far from being a well-quasi-order, this shows the existence of many upward-closed sets of continuous functions with no finite basis. This contrasts drastically with the existence of finite bases for some important classes such as: the discontinuous functions (\cref{prop:contvsdiscont}), the functions that are not $\sigma$-continuous with closed witnesses (\cite{soleckidecomposing,carroymiller2}); the non $\sigma$-continuous functions (\cite{pawlikowski2012decomposing}), the non Baire class $1$ functions (\cite{carroymiller2}). This motivates us to make a conjecture and formulate a general problem:

		\begin{conjecture}
			For all $\alpha<\omega_1$, the class of functions that are not Baire class $\alpha$ admits a finite basis for topological embeddability.
		\end{conjecture}
		
		\begin{question}
			Which classes of functions admit a finite basis?
		\end{question}
		
		Finally, the situation described by \cref{prop:contvsdiscont} suggests an even more specific question:
		
		\begin{question}
			Is there, continuous functions aside, other examples of a class of functions that is both
			defined as the downward closure of finitely-many functions with respect to embeddability and
			whose complement has a finite basis?
		\end{question}
		
		\subsection{Acknowledgments}
		We would like to thank the Descriptive Set Theory group at Paris 6 and the Turin logic group for their patience and careful attention during the seminar presentation of the material exposed in this paper.

		\printbibliography
		
		\medskip

	\end{document}